\documentclass{article}

\usepackage[bf]{caption}

\usepackage
{
    graphicx,
    amssymb,
    amsmath,
    amsthm,
    xcolor,
    braket,
    mathtools,
    authblk,
    xfrac,
    tikz
}

\usepackage[colorlinks,
            pdffitwindow=false,
            plainpages=false,
            pdfpagelabels=true,
            pdfpagemode=UseOutlines,
            pdfpagelayout=SinglePage,
            bookmarks=false,
            colorlinks=true,
            hyperfootnotes=false,
            linkcolor=blue,
            citecolor=green!50!black]{hyperref}

\usepackage[hmargin=2cm,vmargin=2.5cm]{geometry}

\usepackage[labelformat=simple]{subcaption}

\DeclareMathAlphabet{\mathpzc}{OT1}{pzc}{m}{it}

\newcommand\xqed[1]{\leavevmode\unskip\penalty9999 \hbox{}\nobreak\hfill \quad\hbox{#1}}
\newcommand{\exampleSymbol}{\xqed{$\triangle$}}

\newtheorem{theorem}{Theorem}[section]

\newtheorem{definition}[theorem]{Definition}
\theoremstyle{definition}
\newtheorem{example}[theorem]{Example}
\newtheorem{remark}[theorem]{Remark}

\title{Tensor-generated fractals: Using tensor decompositions for creating self-similar patterns}
\author[1]{Patrick Gel\ss}
\author[1,2]{Christof Sch\"utte}
\affil[1]{Department of Mathematics and Computer Science, Freie Universit\"at Berlin, Germany}
\affil[2]{Zuse Institute Berlin, Germany}

\begin{document}
\maketitle

\begin{abstract}

The term \emph{fractal} describes a class of complex structures exhibiting self-similarity across different scales. Fractal patterns can be created by using various techniques such as finite subdivision rules and iterated function systems. In this paper, we will present a method for the construction of geometric fractals that exploits Kronecker products and tensor decompositions, which can be regarded as a generalization of matrix factorizations. We will show how to create several well-known examples for one-, two-, and three-dimensional self-similar structures. Additionally, the proposed method will be extended to the construction of fractals in arbitrary dimensions.

\end{abstract}

\section{Introduction}

Around 1960, the mathematician L.~Richardson observed that the lengths of international borders and coastlines differed widely depending on different sources. In particular, he deduced from the available data that the measured length of a coastline strongly depends on the unit lengths used for measuring and (empirically) increases without limit as the unit length decreases towards 0. This observation was later called the \emph{coastline paradox}. However, as Richardson stated in \cite{RICHARDSON1961}, the measured length $L$ can be approximately expressed as 
\begin{equation*}
  L(r) = \lambda \cdot r^m,
\end{equation*}
where $r \in \mathbb{R}$ is the unit length and $\lambda, m \in \mathbb{R}$ are specific constants. A few years later, B.~Mandelbrot read his publications and rewrote Richardson's equation by replacing $L(r)$ with $N(r) \cdot r$, where $N(r)$ is the number of used units, and by replacing $m$ with $1-D$, i.e.
\begin{equation*}
  N(r) = \lambda \cdot r^{-D} \quad \Leftrightarrow \quad D = \frac{\ln(\lambda) - \ln(N(r))}{\ln(r)},
\end{equation*}
see \cite{MANDELBROT1967}. Mandelbrot interpreted the constant $D \in \mathbb{R}$ as a generalized dimension by comparing coastlines with so-called \emph{monster curves}, e.g.~the Koch curve \cite{KOCH1904} or the curve defined by the Weierstrass function \cite{WEIERSTRASS1886}. These curves were called ``monsters'' because their properties were inconsistent with generally accepted principles of mathematics. More precisely, monster curves are continuous but nowhere differentiable. They are early examples of topological objects displaying \emph{self-similarity}, a term which describes the property of an object to be (approximately) similar to parts of itself. In 1975, Mandelbrot coined the word ``\emph{fractal}'' \cite{MANDELBROT1975}, denoting objects whose Hausdorff--Besicovitch dimensions strictly exceeds their topological dimensions, see Section \ref{sec:Fractal Geometry}. Fractals commonly show self-similarity at increasingly small scales. Approximate fractals can be found in many places in our environment. That is, we see patterns in nature which display quasi self-similarity or statistical self-similarity, respectively, at different scales. Besides coastlines, further examples are river networks \cite{RINALDO1993}, tree branches \cite{ZEIDE1991}, and blood vessels \cite{PEITGEN2004}. For more examples of fractal structures in nature, we refer to \cite{MANDELBROT1982}. 

In general, fractals often have various remarkable properties and are therefore of high interest from a mathematical point of view. Since the 19th century, many mathematicians and other scientists have made contributions to the field of fractals and self-similarity including B.~Bolzano \cite{BOLZANO1930}, G.~Peano \cite{PEANO1890}, and D.~Hilbert \cite{HILBERT1891}. Additionally, the ideas of fractal geometry are also used in practical applications such as antenna design \cite{WERNER1999}, image compression \cite{BARNSLEY1993}, and computer graphics \cite{ENCARNACAO1992}. 

In this paper, we will consider fractals arising from iteratively applying mathematical functions to certain sets. Furthermore, we will include the concept of \emph{tensor decompositions} in this work. Tensors are multidimensional generalizations of vectors and matrices represented by arrays with several indices. By using so-called \emph{tensor products}, high-dimensional tensors can be decomposed into networks consisting of several smaller tensors. Over the last years, the interest in tensor decompositions has been growing rapidly since tensor-product approaches can be applied to many different application areas in order to mitigate the curse of dimensionality, see e.g.~\cite{KAZEEV2014, DOLGOV2015, GELSS2017}. Several tensor formats such as the \emph{canonical format} \cite{HITCHCOCK1927}, the \emph{Tucker format} \cite{TUCKER1963,TUCKER1964}, and the \emph{tensor-train format} \cite{OSELEDETS2009,OSELEDETS2009b,OSELEDETS2011} have been introduced and it was shown that various high-dimensional systems can be treated directly in those tensor formats. In particular, the tensor-train format may also be used for compact and storage-efficient representations of large data sets by capturing self-similar structures in the data, see \cite{LARCHER2017}. In this study, we will extend the ideas presented in \cite{XUE1996, HARDY2008, LESKOVEC2010} and use tensor decompositions for the opposite direction, i.e.~for generating self-similar patterns in different dimensions. We will show that various well-known geometric fractals can be created using tensor products and generalized Kronecker products.

The paper is organized as follows: In Section \ref{sec:Fractal Geometry}, we give a brief overview of fractal geometry including basic definitions and some prominent examples. Furthermore, we describe the fractal construction by iterated function systems. In Section \ref{sec:Tensor Decompositions}, we describe the concept of tensors and tensor decompositions, focusing on the tensor-train format. Examples of tensor-generated fractals will be shown in Section \ref{sec:Tensor-Generated Fractals}. In Section \ref{sec:Conclusion}, we will conclude with a short summary and an outlook on future work.

\section{Fractal geometry}
\label{sec:Fractal Geometry}

In mathematics, the dimension of an object essentially identifies the number of degrees of freedom of the elements belonging to that object. Intuitively, the dimension is the minimum number of coordinates required to specify any point of the object. However, as we will explain in this section, there exist different concepts of dimensions specifically defined for different contexts. The intuitive concept of a dimension is given by the \emph{topological dimension}.

\begin{definition}
 The \emph{topological dimension} $D_T$, also called \emph{Lebesgue covering dimension}, of a compact metric space $X \subset \mathbb{R}^d$ is the smallest number $n \in \mathbb{N}$ such that there exists an open cover $(U_i)_{i \in I}$ with $\textrm{diam}(U_i) < \varepsilon$ for all $i \in I$ and for any given $\varepsilon > 0$, where each point $x \in X$ belongs to at most $n+1$ sets in the cover.
\end{definition}
For instance, the topological dimension of a curve (including graphs) is 1, whereas the topological dimension of a surface is 2. In general, a $d$-dimensional Euclidean space has topological dimension $D_T=d$. 

The \emph{fractal dimension}, as introduced by Mandelbrot in 1967, is a special case of the \emph{Hausdorff--Besicovitch dimension} which was already introduced in 1918 by F.~Hausdorff, see \cite{HAUSDORFF1918}. In general, the fractal dimension of a compact metric space $X$ is given by the so-called \emph{box-counting dimension}, cf.~\cite{LIEBOVITCH1989,SARKAR1994}.
\begin{definition}\label{def: fractal dimension}
  The \emph{fractal dimension} $D_F$ of a compact metric space $X \subset \mathbb{R}^d$ is defined as
  \begin{equation}\label{eq: fractal dimension}
    D_F = - \lim_{r \rightarrow 0} \frac{\ln(N(r))}{\ln(r)},
  \end{equation}
  where $N(r)$ is number of ($d$-dimensional) boxes of side length $r$ required to cover the space $X$.
\end{definition}

In order to understand the Definition \ref{def: fractal dimension}, we consider simple geometric objects such as lines, squares or cubes. We scale these objects by scaling factors $r<1$ and denote by $N(r)$ the number of subobjects, which then fit into the given object, see Figure \ref{fig: simple objects}. One can see that we obtain the same relation as defined in \eqref{eq: fractal dimension}, i.e.
\begin{equation}\label{eq: simple dimension}
  N(r) = r^{-D} \quad \Leftrightarrow \quad D = -\frac{\ln(N(r))}{\ln(r)}.
\end{equation}
The only difference between \eqref{eq: fractal dimension} and \eqref{eq: simple dimension} is that we do not take the limit as $r \rightarrow 0$ in \eqref{eq: simple dimension} since the result is the same for all scaling factors. Thus, we see that $D_F$ can indeed be interpreted as a generalized form of the intuitive concept of a dimension. 

\begin{figure}[htb]
    \centering
    \begin{subfigure}[b]{0.3\textwidth}
        \centering
        \includegraphics[width=80px]{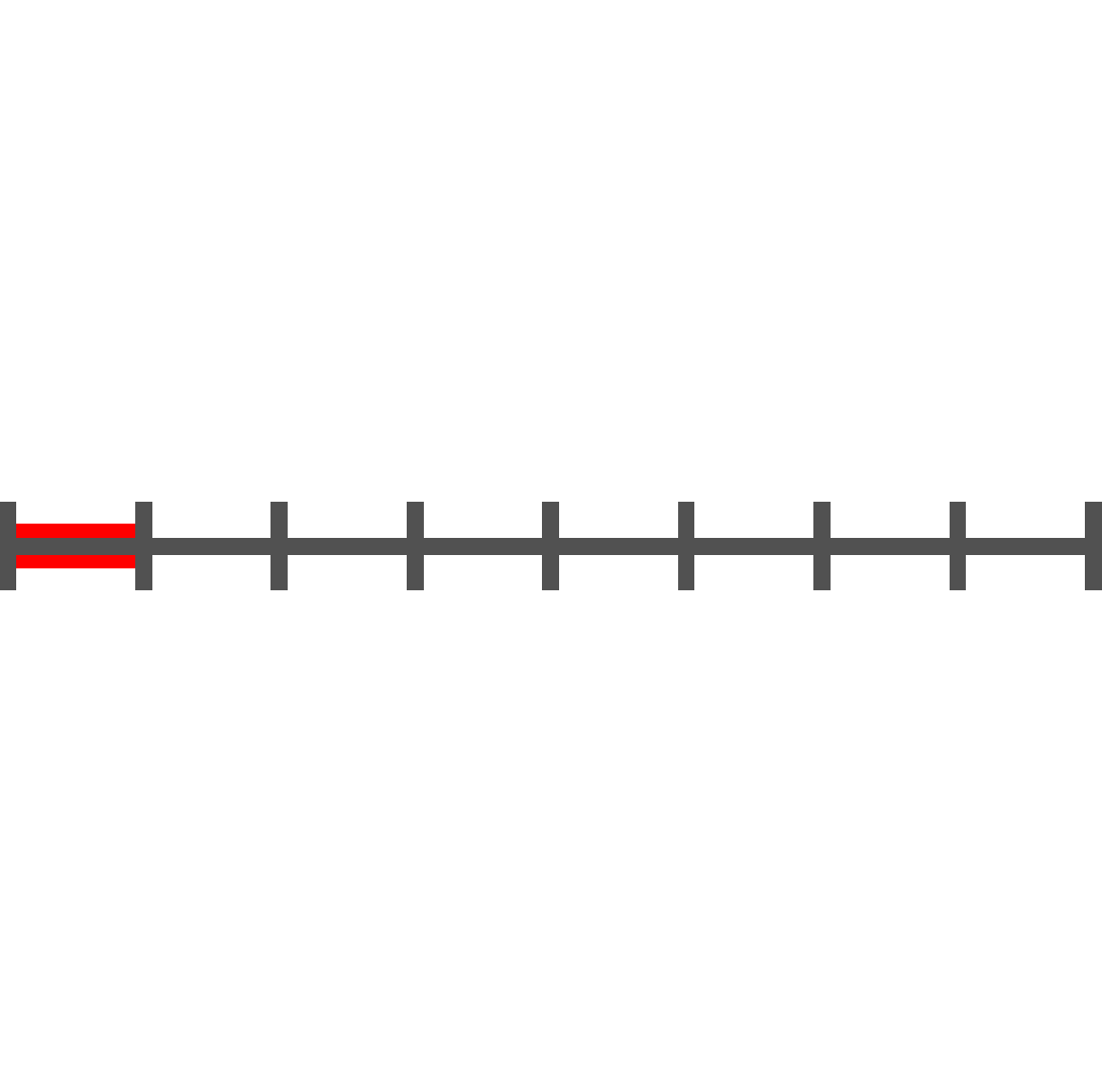}
        \caption{}
    \end{subfigure}
    \hfill
    \begin{subfigure}[b]{0.3\textwidth}
        \centering
        \includegraphics[width=80px]{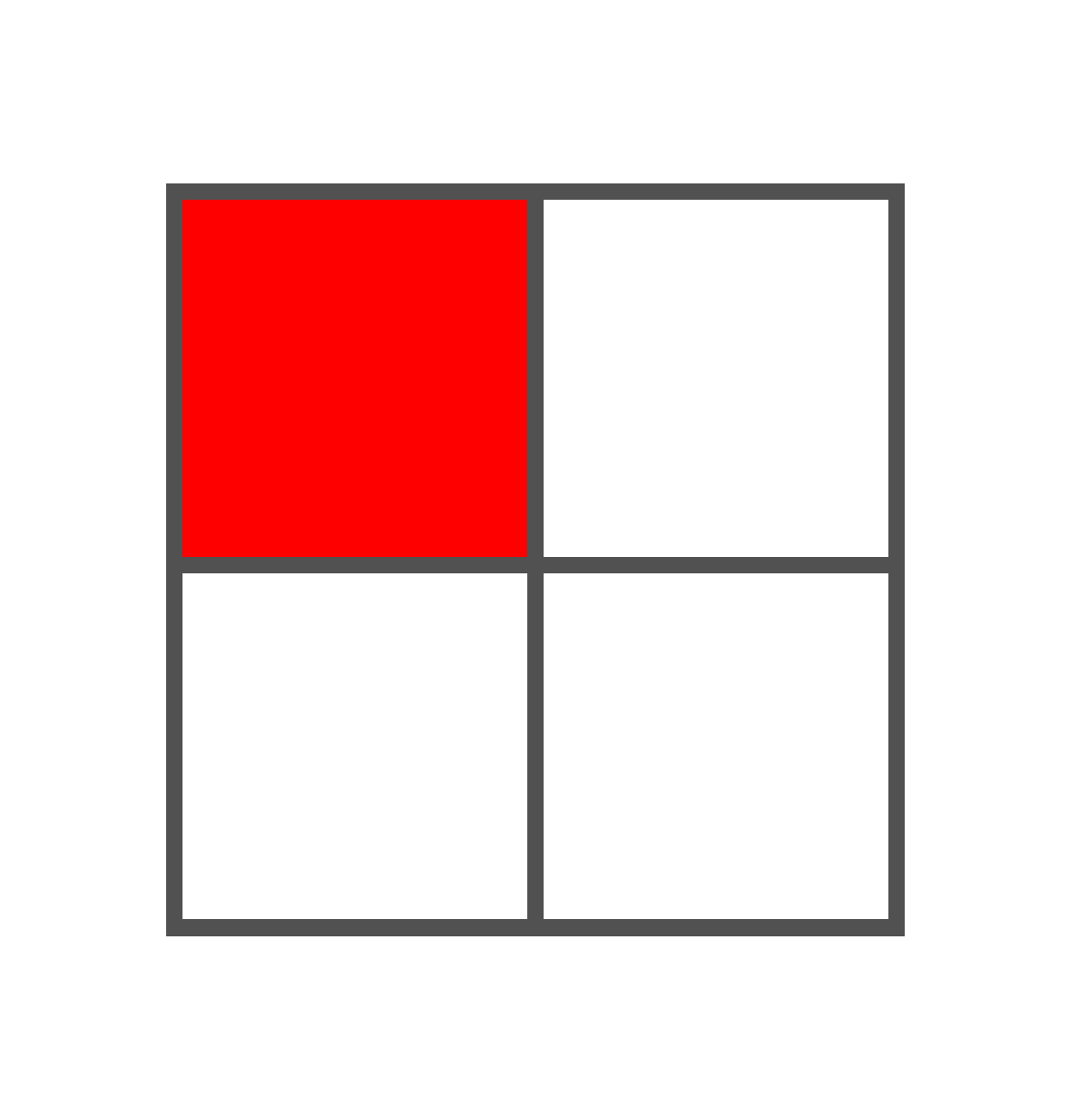}
        \caption{}
    \end{subfigure}
    \hfill
    \begin{subfigure}[b]{0.3\textwidth}
        \centering
        \includegraphics[width=80px]{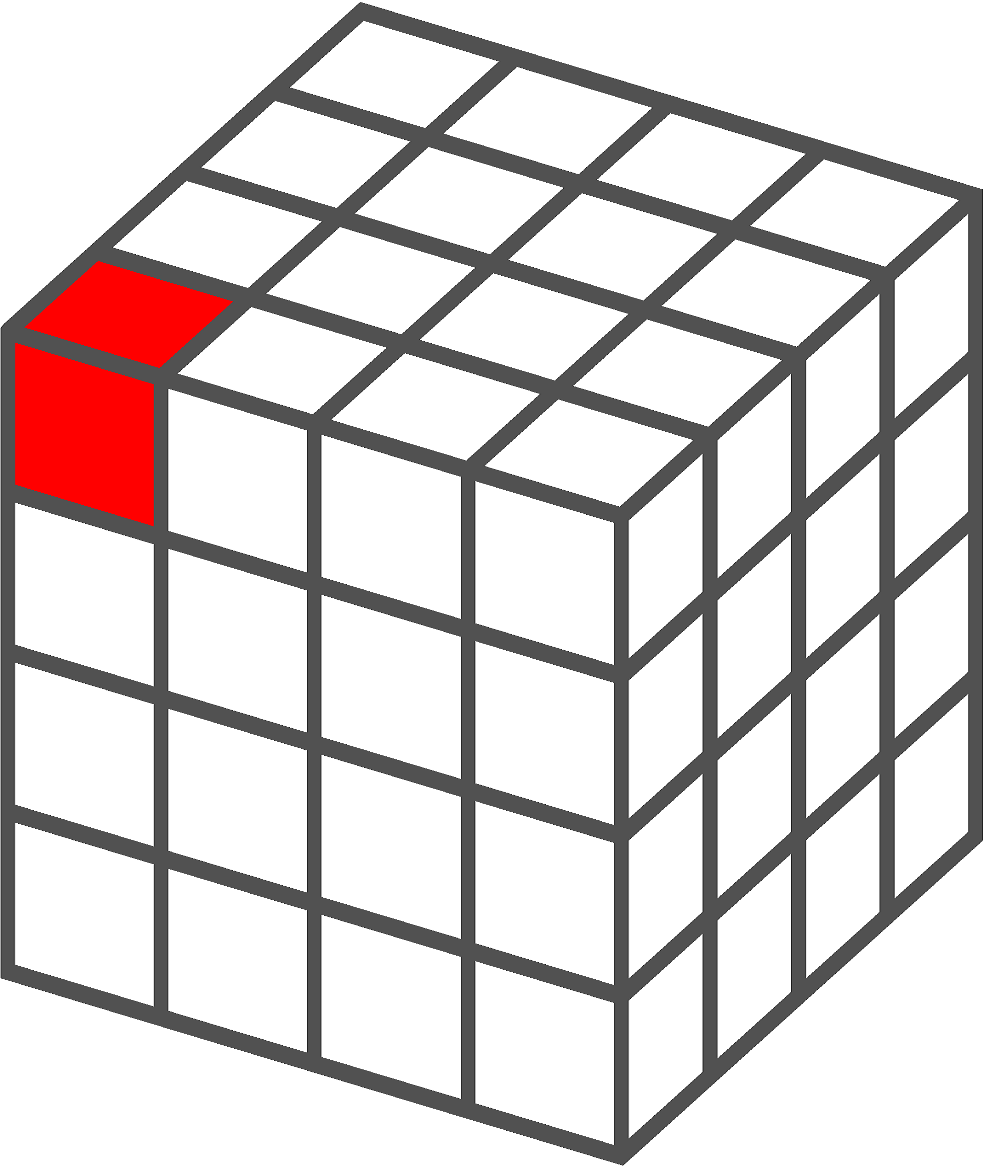}
        \caption{}
    \end{subfigure}
    \caption{Scaling of simple geometric objects: (a) Line segment ($D_T = D_F  = 1$) scaled by a factor of $1/8$. (b) Square ($D_T = D_F = 2$) scaled by a factor of $1/2$. (c) Cube ($D_T=D_F=3$) scaled by a factor of $1/4$. The respective dimensions satisfy \eqref{eq: simple dimension}.}
    \label{fig: simple objects}
\end{figure}

The fundamental definition of a fractal was stated by Mandelbrot in \cite{MANDELBROT1975} by comparing the topological dimension and the Hausdorff--Besicovitch dimension of a given object.
\begin{definition}
 A \emph{fractal} is a compact metric space $X \subset \mathbb{R}^d$ whose fractal dimension strictly exceeds its topological dimension, i.e.~$D_F > D_T$.
\end{definition}

Prominent examples of fractal structures which arise from simple mathematical equations are Julia sets and the Mandelbrot set, see \cite{MANDELBROT1982}. In general, Julia sets are non-empty, uncountable, compact, and perfect, i.e.~they have no isolated points. The boundary of the Mandelbrot set is a fractal curve with fractal dimension 2. That is, in terms of fractal geometry, it is a curve in a two-dimensional space which is as complicated as it can be. Despite being a one-dimensional object, it behaves like a surface.

In this paper, we focus on strictly geometric fractals. A simple and early example is the Cantor set, see Figure \ref{fig: Cantor}, which was introduced by G.~Cantor in 1883 \cite{CANTOR1883}. It can be constructed by using a finite subdivision rule, i.e.~starting with the unit interval $\left[0,1\right]$, the open middle third of every line segment is removed in each iteration step. The Cantor set is then defined as the limit set obtained by repeating the subdivision infinitely many times. Just like the Julia sets, the Cantor set is non-empty, uncountable, compact, and perfect. Furthermore, it is nowhere dense and totally disconnected. Since the Cantor set is a null set, its topological dimension is $0$. However, its fractal dimension can be calculated by
\begin{equation*}
  D_F = - \lim_{r \rightarrow 0}\frac{\ln(N(r))}{\ln(r)} = - \lim_{k \rightarrow \infty}\frac{\ln(2^k)}{\ln\left(\left(\frac{1}{3}\right)^k\right)} = \frac{\ln(2)}{\ln(3)} \approx 0.6309,
\end{equation*}
since the line segments are scaled by a factor of $1/3$ at each iteration step and their number is doubled, cf.~Figure~\ref{fig: Cantor}.

\begin{figure}[htb]
    \centering
    \begin{subfigure}[b]{0.9\textwidth}
        \centering
        \raisebox{-1.5px}{\makebox[1cm][c]{\small\textbf{(a)}}}
        \includegraphics[width=300px]{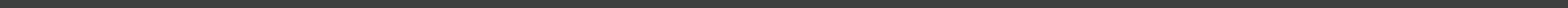}
        \caption*{}
    \end{subfigure}\\
    \begin{subfigure}[b]{0.9\textwidth}
        \centering
        \raisebox{-1.5px}{\makebox[1cm][c]{\small\textbf{(b)}}}
        \includegraphics[width=300px]{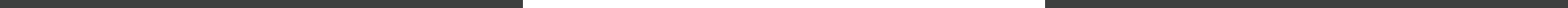}
        \caption*{}
    \end{subfigure}\\
    \begin{subfigure}[b]{0.9\textwidth}
        \centering
        \raisebox{-1.5px}{\makebox[1cm][c]{\small\textbf{(c)}}}
        \includegraphics[width=300px]{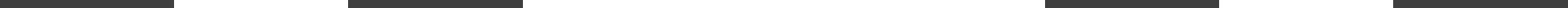}
        \caption*{}
    \end{subfigure}\\
    \begin{subfigure}[b]{0.9\textwidth}
        \centering
        \raisebox{-1.5px}{\makebox[1cm][c]{\small\textbf{(d)}}}
        \includegraphics[width=300px]{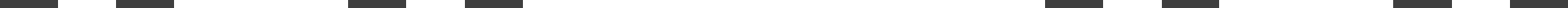}
        \caption*{}
    \end{subfigure}
    \caption{Construction of the Cantor set: The first four iteration steps of the fractal construction are shown. In each iteration step, the open middle third of every line segment is removed.}
    \label{fig: Cantor}
\end{figure}

\subsection{Iterated function systems}
\label{sec:Iterated Function Systems}

Fractal patterns can also be constructed by using iterated function systems (IFS), which are finite sets of contractive functions $f_i : X \rightarrow X$, $i = 1, \dots , m$, on a complete metric space. That is, there exists a real number $\kappa_i$ for any $i \in \{1, \dots , m\}$ with $0 < \kappa_i < 1$ such that 
 \begin{equation*}
  \Delta(f_i(x), f_i(y)) \leq \kappa_i \cdot \Delta (x,y),
\end{equation*}
for all $x,y \in X$. Here, $\Delta(\cdot,\cdot)$ denotes the metric on $X \subset \mathbb{R}^d$. For the fractal construction by using an IFS, we set $X$ to $\left[ 0,1 \right]^{\times d}$ with $1 \leq d \leq 3$. That is, the space $X$ is either the unit interval, the unit square, or the unit cube. 
\begin{definition}
  Let $Y \subseteq X$ be a subset of the metric space $X$. The \emph{Hutchinson operator} corresponding to the IFS $\{f_1 , \dots , f_m\}$ is given by
  \begin{equation*}
    H(Y) = \bigcup_{i = 1}^m f_i (Y).
  \end{equation*}
  The fixed set of the Hutchinson operator, i.e.~$Z \subseteq X$ with $H(Z) = Z$, is called the \emph{IFS attractor}.
\end{definition}

As a consequence of the Banach fixed-point theorem, the IFS attractor $Z$ is uniquely defined for any given (countable) IFS, see \cite{SECELEAN2012}. Furthermore, it can be represented as a limit set of the iterative application of $H$ to an arbitrary compact subset $Y \subseteq X$ with $Y \neq \varnothing$, i.e.
 \begin{equation*}
  Z = \lim_{k \rightarrow \infty} H^k(Y).
\end{equation*}

IFS attractors often exhibit self-similarity over an infinite range of scales. The Cantor set, which was already mentioned in the previous section, see Figure \ref{fig: Cantor}, can also be constructed by using an IFS instead of a finite subdivision rule. The corresponding set of functions is given by
\begin{equation*}
  \left\lbrace  f_1 (x)  =  \frac{1}{3} x,~  f_2(x) = \frac{1}{3} x + \frac{2}{3}  \right\rbrace,
\end{equation*}
with $x \in \left[0,1\right]$. The functions iteratively scale and shift a given initial subset of $\left[ 0,1\right]$. If we, for the sake of simplicity, start with the whole unit interval as the initial set, we get exactly the same construction steps as shown in Figure \ref{fig: Cantor}. 

\begin{figure}[htb]
    \centering
    \begin{subfigure}[b]{0.24\textwidth}
        \centering
        \includegraphics[width=80px]{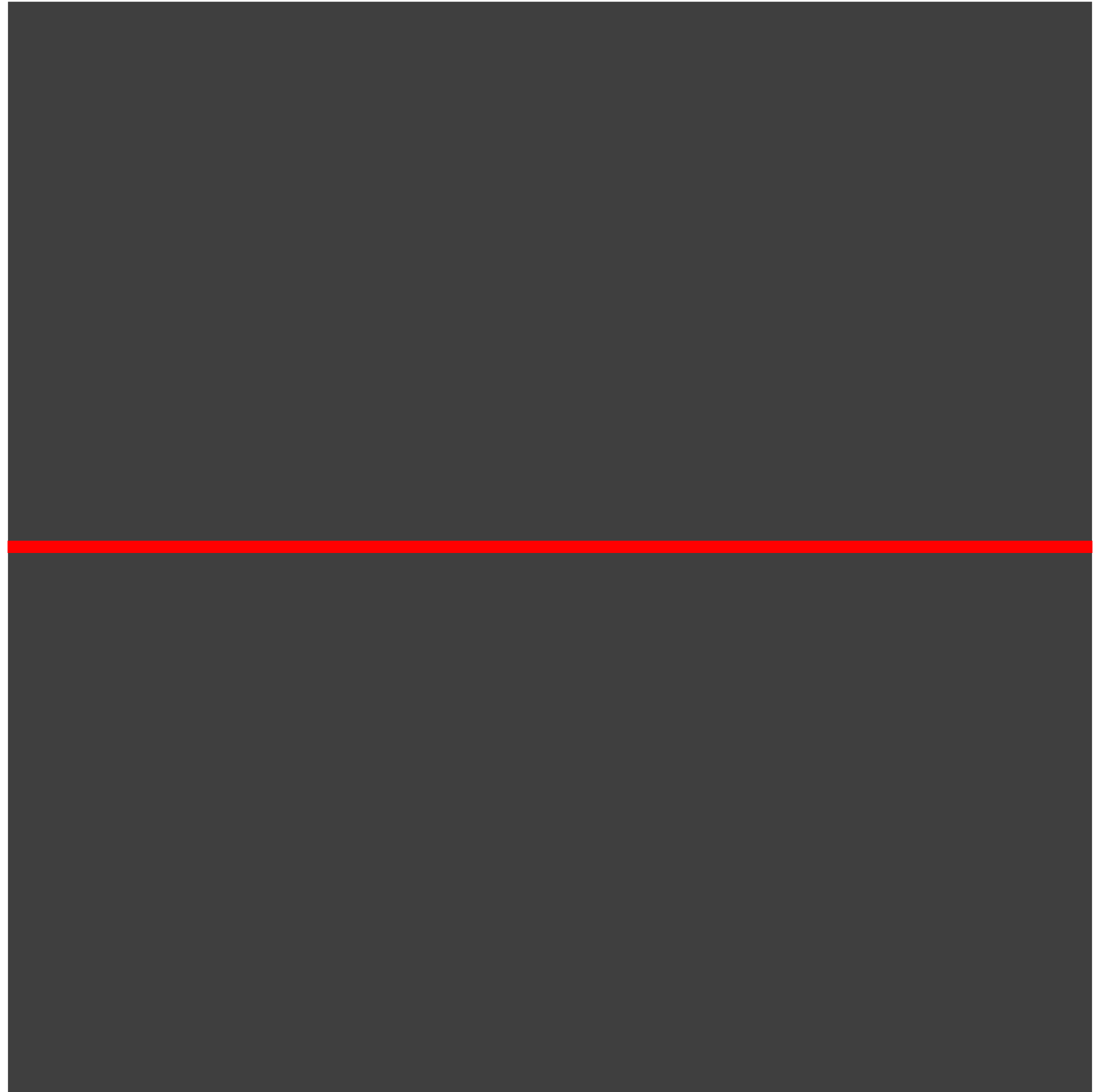}
        \caption{}
    \end{subfigure}
    \hfill
    \begin{subfigure}[b]{0.24\textwidth}
        \centering
        \includegraphics[width=80px]{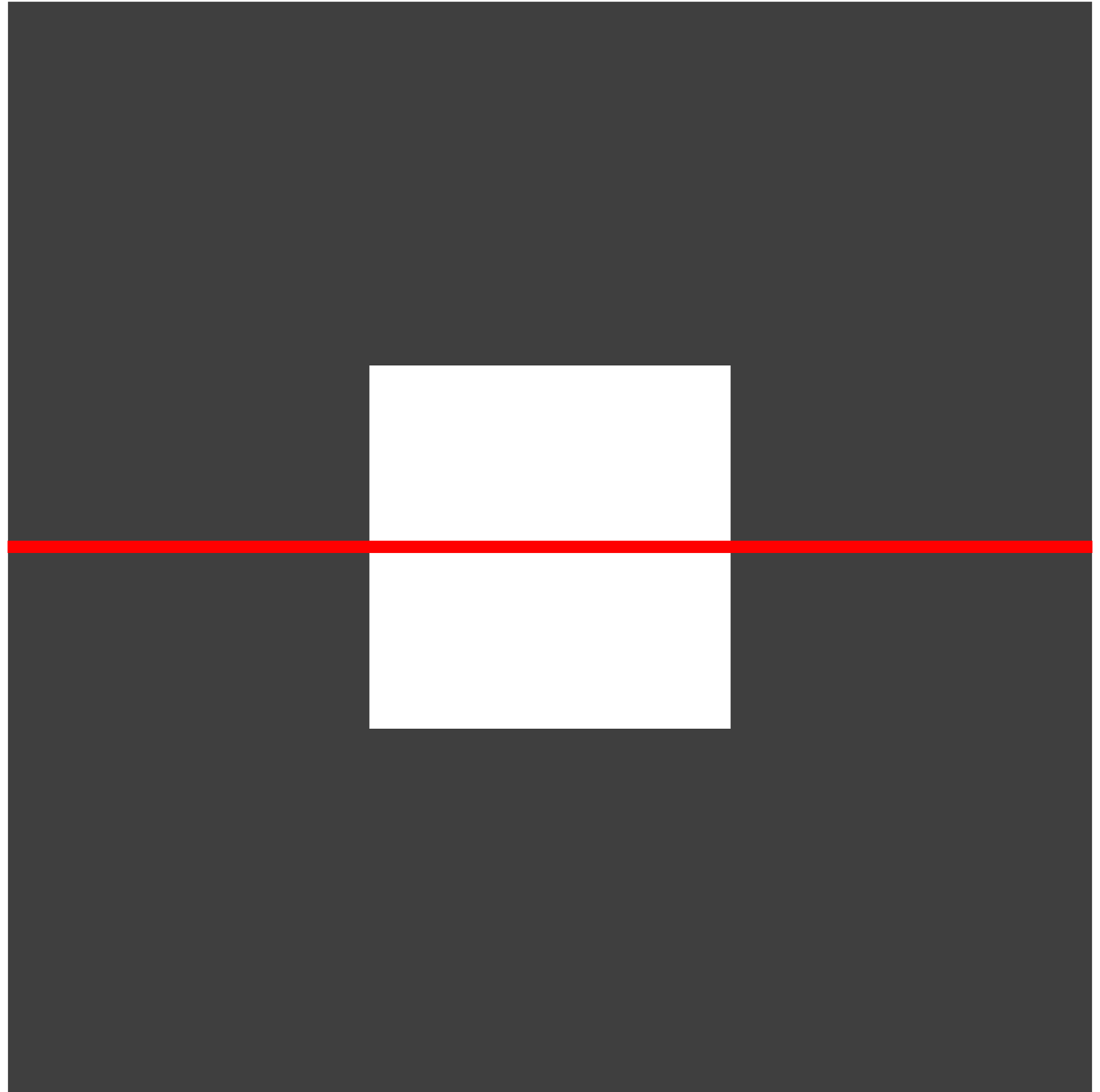}
        \caption{}
    \end{subfigure}
    \hfill
    \begin{subfigure}[b]{0.24\textwidth}
        \centering
        \includegraphics[width=80px]{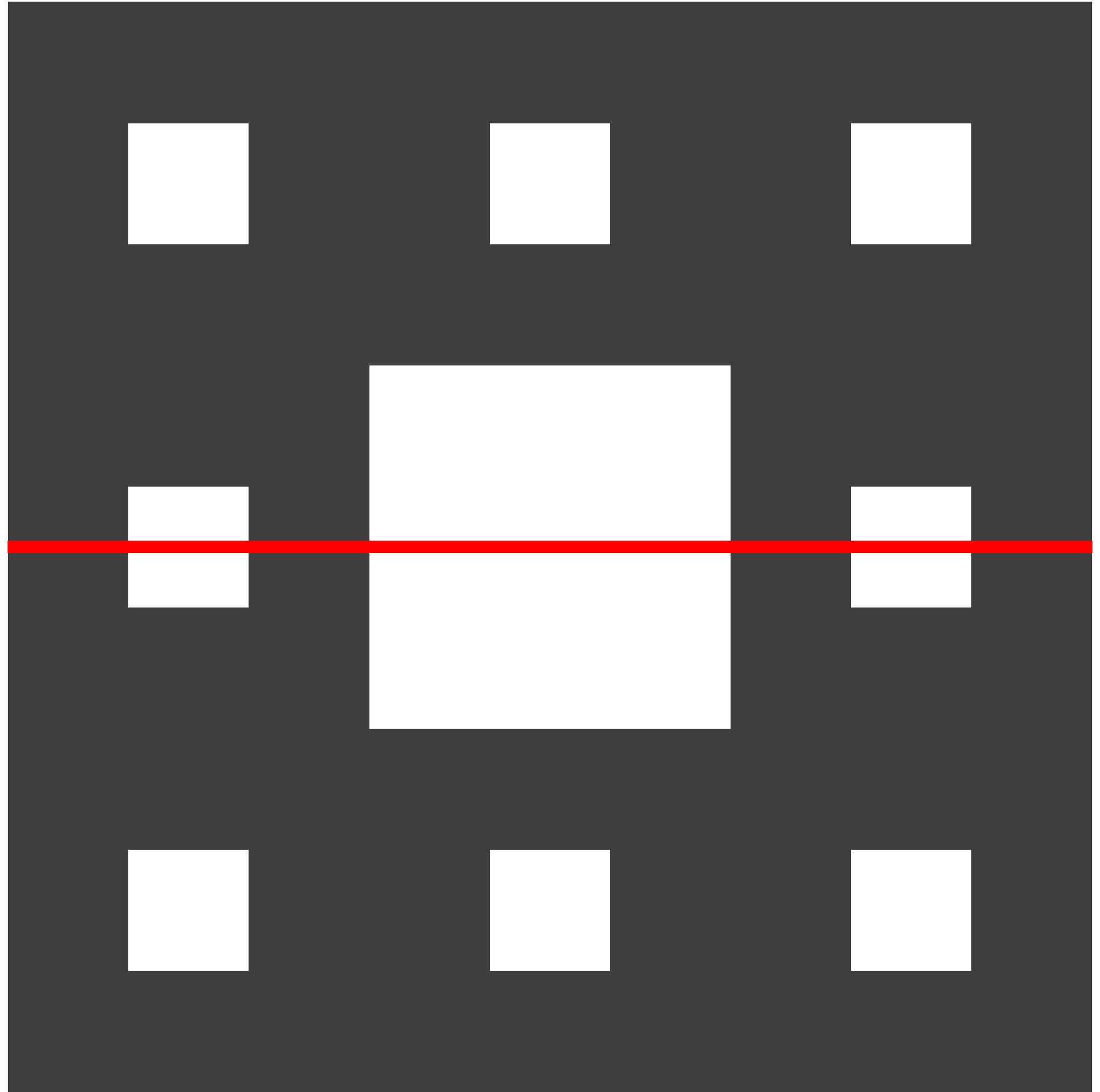}
        \caption{}
    \end{subfigure}
    \hfill
    \begin{subfigure}[b]{0.24\textwidth}
        \centering
        \includegraphics[width=80px]{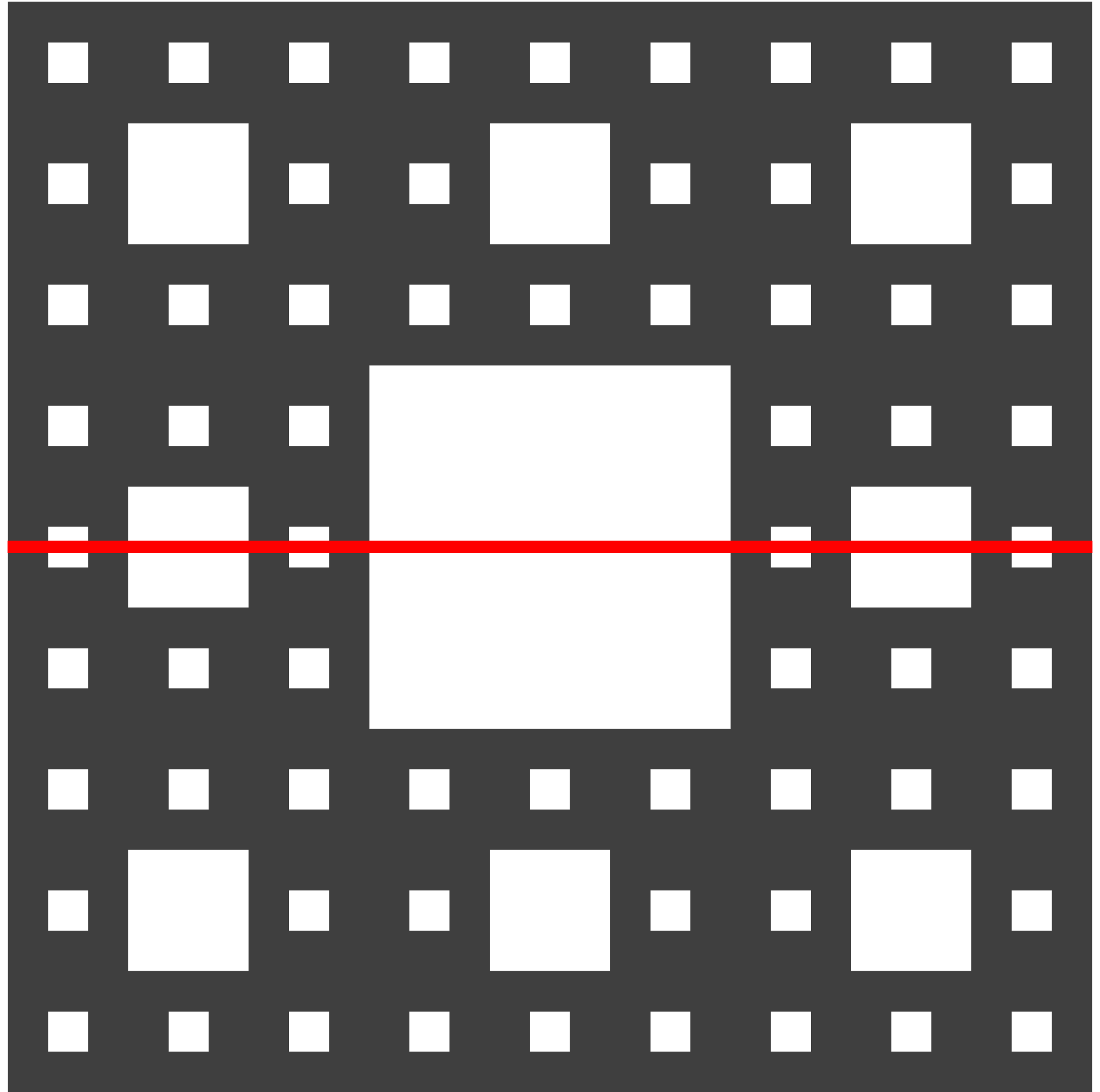}
        \caption{}
    \end{subfigure}
    \caption{Construction of the Sierpinski carpet: The first four iteration steps of the fractal construction are shown. In each iteration step, the open central part of every subsquare is removed. The intersections with the red lines correspond to the construction steps of the Cantor set shown in Figure \ref{fig: Cantor}.}
    \label{fig: sierpinski}
\end{figure}

Two other famous IFS fractals we will focus on in this paper are the Sierpinski carpet, introduced by W. Sierpi{\'n}ski in 1916 \cite{SIERPINSKI1916}, and the Menger sponge, introduced by K. Menger in 1926 \cite{MENGER1916}. Both can be seen as higher-dimensional generalizations of the Cantor set. That is, the Sierpinski carpet is a two-dimensional counterpart of the Cantor set. The corresponding IFS can be found in Appendix \ref{app: IFS - sierpinski}. Figure \ref{fig: sierpinski} shows the first four iteration steps of the construction of the Sierpinski carpet and its connection to the Cantor set. The Menger sponge is a three-dimensional generalization of the Sierpinski carpet and the Cantor set, respectively. See Figure \ref{fig: menger} for an illustration of the construction and Appendix \ref{app: IFS - menger} for the corresponding IFS. Each face of the Menger sponge is a Sierpinski carpet and the intersection of the sponge with any midline of the faces is a Cantor set. The Sierpinski carpet as well as the Menger sponge have in fact topological dimension of 1, meaning both are curves in two and three dimensions, respectively. The fractal dimension of the Sierpinski carpet is $D_F = \ln(8) / \ln(3) \approx 1.8928$ and the fractal dimension of the Menger sponge is $D_F = \ln(20) / \ln(3) \approx 2.7268$. These calculations become clear by considering the number of residual subsquares/subcubes in each iteration step. In both cases the subsquares/subcubes are scaled by a factor of $1/3$, determining the denominators in \eqref{eq: fractal dimension}. 

\begin{figure}[htb]
    \centering
    \begin{subfigure}[b]{0.24\textwidth}
        \centering
        \includegraphics[width=80px]{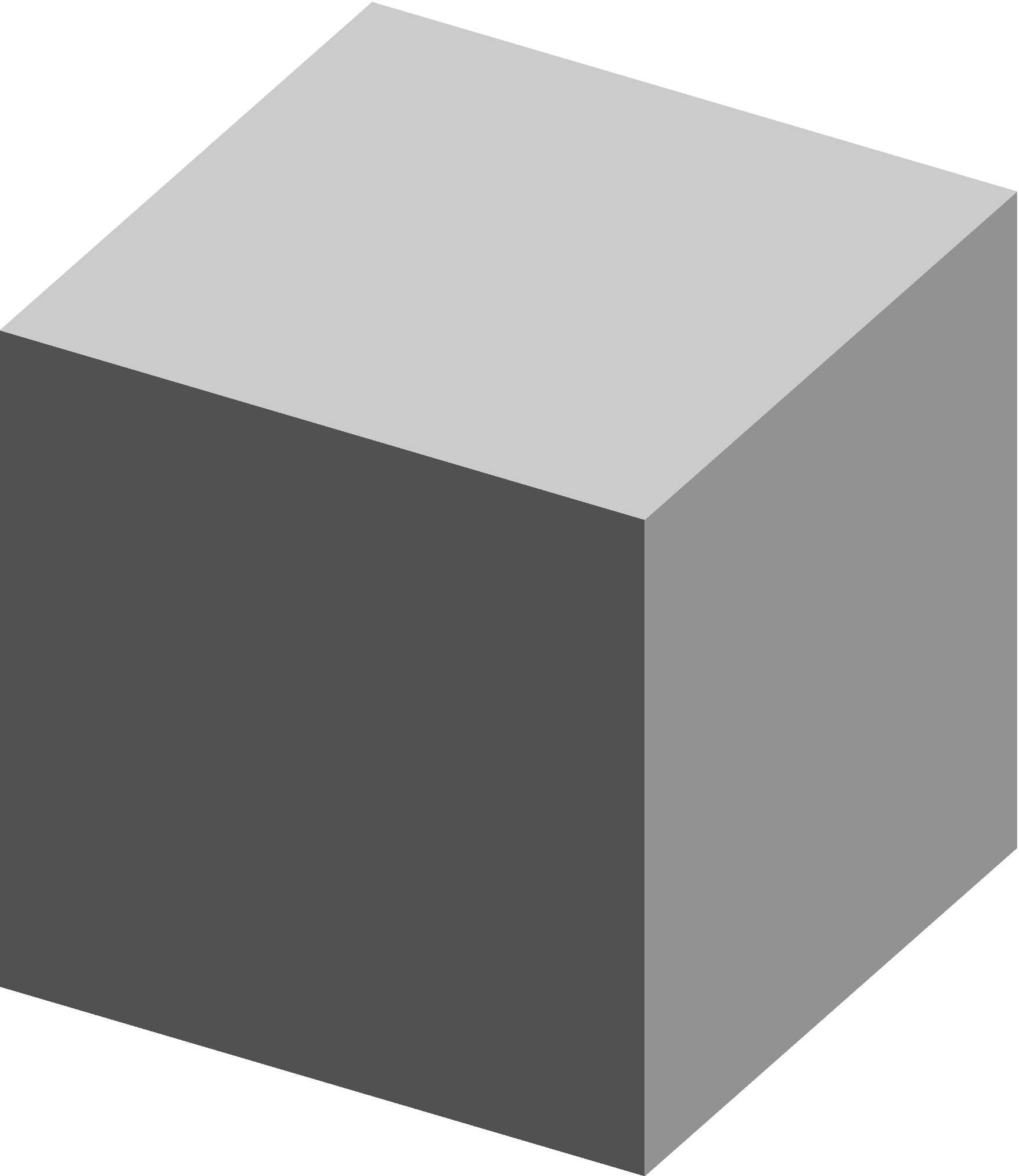}
        \caption{}
    \end{subfigure}
    \hfill
    \begin{subfigure}[b]{0.24\textwidth}
        \centering
        \includegraphics[width=80px]{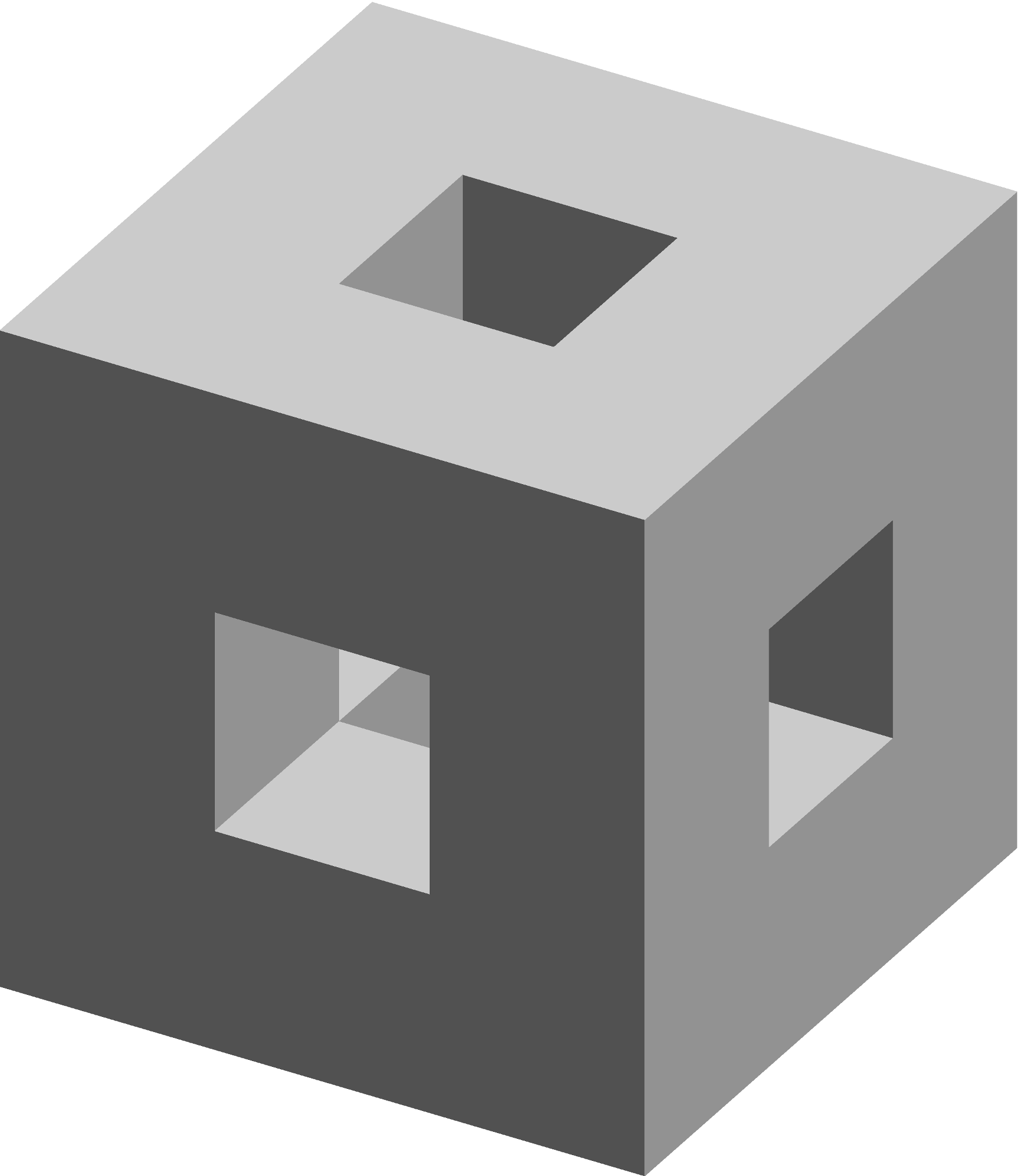}
        \caption{}
    \end{subfigure}
    \hfill
    \begin{subfigure}[b]{0.24\textwidth}
        \centering
        \includegraphics[width=80px]{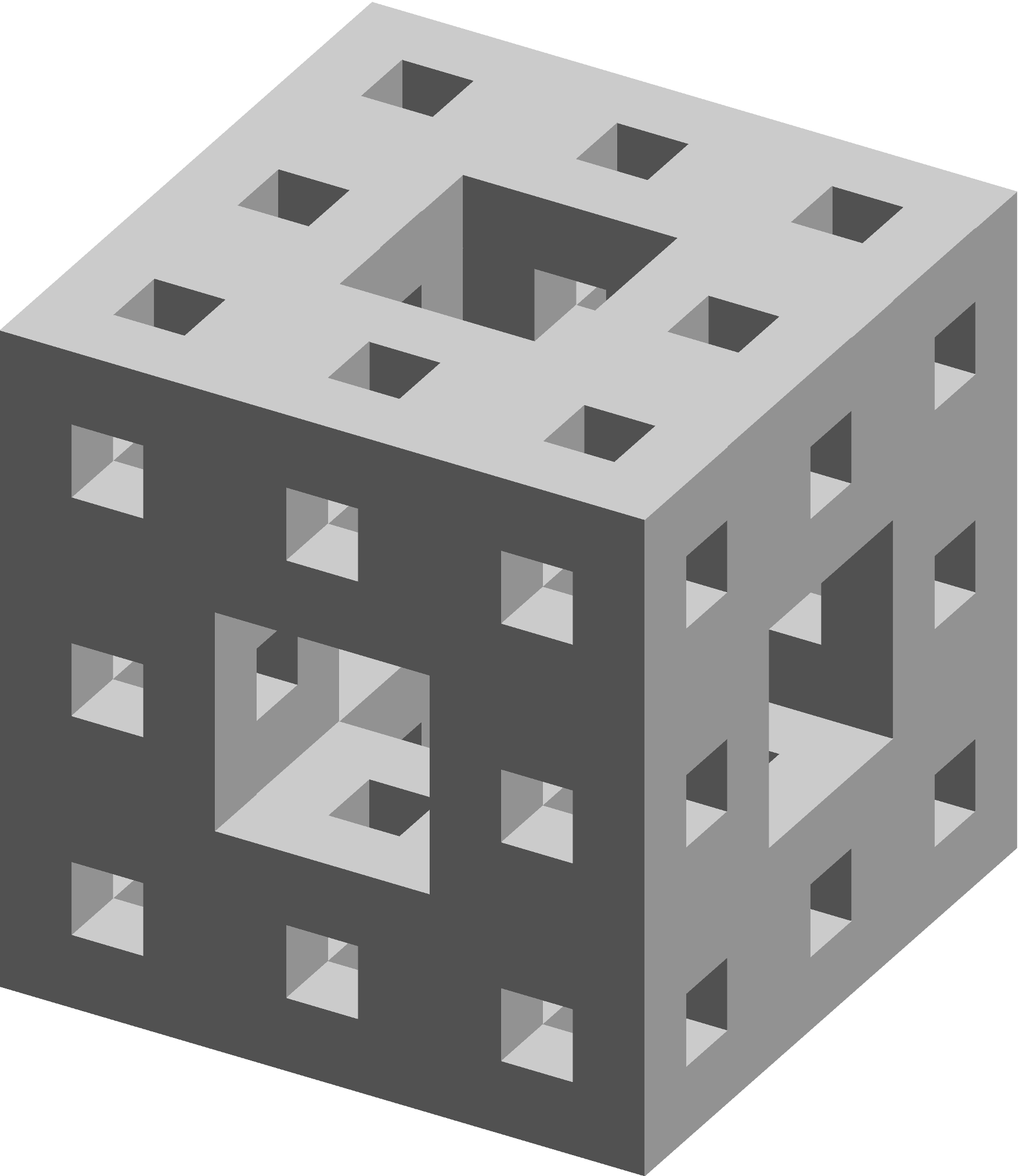}
        \caption{}
    \end{subfigure}
    \hfill
    \begin{subfigure}[b]{0.24\textwidth}
        \centering
        \includegraphics[width=80px]{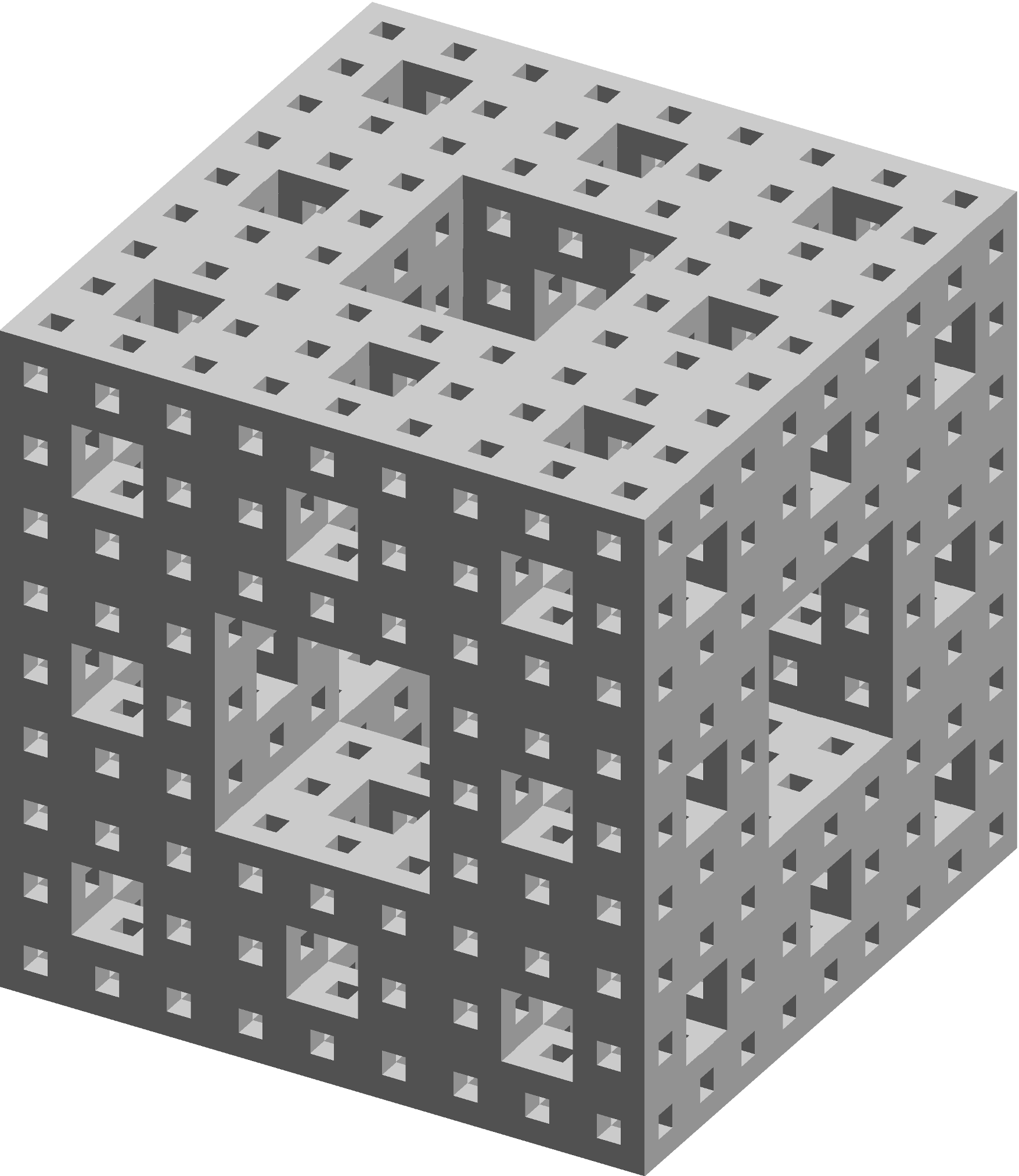}
        \caption{}
    \end{subfigure}
    \caption{Construction of the Menger sponge: The first four iteration steps of the fractal construction are shown. In each iteration step, every subcube is divided into 27 smaller cubes. Subsequently, the central smaller cubes in the middle of each face and in the center of the subcube are then removed.}
    \label{fig: menger}
\end{figure}

\section{Tensor decompositions}
\label{sec:Tensor Decompositions}

A tensor is a multidimensional array $\mathbf{T} \in \mathbb{R}^{n_1 \times \dots \times n_d}$ whose entries are indexed by $ \mathbf{T}_{x_1, \dots, x_d} $. The number $d$ is called the \emph{order} of the tensor $\mathbf{T}$. If we fix certain indices, colons are used to indicate the free modes (cf.~\textsc{Matlab} colon notation), e.g.~for a tensor $\mathbf{T} \in \mathbb{R}^{n_1  \times \dots \times  n_d}$, we obtain 
\begin{equation*}
\mathbf{T} _{x_1 , : , x_3 , : , x_5, \ldots , x_d} \in \mathbb{R}^{n_2 \times n_4} \quad \textrm{and} \quad \mathbf{T} _{x_1 , : , \ldots , : , x_d} \in \mathbb{R}^{n_2 \times  \dots \times n_{d-1}}.
\end{equation*}

A tensor of order 1 is a vector, a tensor of order 2 is a matrix, and a tensor of order 3 can be imagined as layers of matrices. In order to visualize tensors $\mathbf{T} \in \mathbb{R}^{n_1 \times n_2 \times 3}$, we list the single layers of the tensor in a row, i.e.

\begin{center}
\begin{tikzpicture}
\node[] at (-1.15,0.53) {$\mathbf{T}~=$};
\def\da{0.38}
\def\db{0.475}
\node (1) [draw=lightgray, rectangle, fill=white] at (2*\da,2*\db) {$\mathbf{T}_{\textrm{:},\textrm{:},3}$};
\node (2) [draw=lightgray, rectangle, fill=white] at (1*\da,1*\db) {$\mathbf{T}_{\textrm{:},\textrm{:},2}$};
\node (3) [draw=lightgray, rectangle, fill=white] at (0*\da,0*\db) {$\mathbf{T}_{\textrm{:},\textrm{:},1}$};
\draw[dotted] (1.north east)--(3.north east);
\draw[dotted] (1.north west)--(3.north west);
\draw[dotted] (1.south east)--(3.south east);
\node[align = left, anchor = west] at (1.5,0.53) {$=~\left( \, \mathbf{T}_{\textrm{:},\textrm{:},1} \,|\, \mathbf{T}_{\textrm{:},\textrm{:},2} \,|\, \mathbf{T}_{\textrm{:},\textrm{:},3} \,\right),$};
\end{tikzpicture}
\end{center}
with $\mathbf{T}_{:,:,i} \in \mathbb{R}^{n_1 \times n_2}$ for $i=1,2,3$.

With the aim to mitigate the curse of dimensionality, i.e.~the exponential growth of the memory consumption of a tensor, various tensor formats have been proposed over the last years. The common basis of these formats is the tensor product. 

\begin{definition} \label{def:outer product}
    The \emph{tensor product} of two tensors $\mathbf{T} \in \mathbb{R}^{m_1 \times \dots \times m_d}$ and $\mathbf{U} \in \mathbb{R}^{n_1 \times \dots \times n_e}$ defines a tensor $\mathbf{T} \otimes \mathbf{U} \in \mathbb{R}^{(m_1 \times \dots \times m_d) \times (n_1 \times \dots \times n_e)}$ with
    \begin{equation*}
        \left(\mathbf{T} \otimes \mathbf{U} \right)_{x_1, \dots, x_d, y_1, \dots, y_e } = \mathbf{T}_{x_1, \dots, x_d} \cdot \mathbf{U}_{y_1, \dots, y_e},
    \end{equation*}
    where $1 \leq x_k \leq m_k$ for $k = 1, \dots, d$ and $1 \leq y_k \leq n_k$ for $k = 1, \dots, e$.
\end{definition}

The tensor product enables the decomposition of high-dimensional tensors into several smaller tensors. It is a bilinear map, meaning that, if we fix one of the tensors, we obtain a linear map on the space where the other tensor lives. In 1927, F.~Hitchcock presented the idea of expressing a tensor as the sum of a finite number of so-called \emph{rank-one tensors} \cite{HITCHCOCK1927}, i.e.
\begin{equation}\label{eq: rank-one tensor}
  \mathbf{T} = \sum_{k = 1}^{r} \left( \mathbf{T}^{(1)} \right)_{k,:} \otimes \dots \otimes \left( \mathbf{T}^{(d)} \right)_{k,:},
\end{equation}
with \emph{cores} $\mathbf{T}^{(i)} \in \mathbb{R}^{r \times n_i}$ for $i=1, \dots, d$. The parameter $r$ is called the \emph{rank} of the decomposition. The above format is the initial concept of tensor decompositions. In fact, any tensor can be represented by a linear combination of tensor products of vectors as in \eqref{eq: rank-one tensor}. However, the number of required rank-one tensors plays an important role. For more information about canonical tensors, we refer to~\cite{KOLDA2009}.

Another fundamental operation is the \emph{Kronecker product}, which has a close relation to the tensor product in terms of vectorizations and matricizations, see e.g.~\cite{COHEN2015b,CICHOCKI2016}. Usually, the Kronecker product is only applied to vectors and matrices. However, we generalize this operation to tensors with arbitrary order.

\begin{definition} \label{def:Kronecker product}
    The \emph{generalized Kronecker product} of two tensors $\mathbf{T} \in \mathbb{R}^{m_1 \times \dots \times m_d}$ and $\mathbf{U} \in \mathbb{R}^{n_1 \times \dots \times n_d}$ defines a tensor $\mathbf{T} \boxtimes \mathbf{U} \in \mathbb{R}^{(m_1 \cdot n_1) \times \dots \times (m_d \cdot n_d)}$ with
    \begin{equation*}
        \left(\mathbf{T} \boxtimes \mathbf{U} \right)_{n_1 (x_1 -1)+y_1, \dots, n_d (x_d -1)+y_d} = \mathbf{T}_{x_1, \dots, x_d} \cdot \mathbf{U}_{y_1, \dots, y_d},
    \end{equation*}
    where $1 \leq x_k \leq m_k$ and $1 \leq y_k \leq n_k$ for $k = 1, \dots, d$.
\end{definition}

Note that some of the dimensions $m_1, \dots , m_d$ and $n_1 , \dots , n_d$ in Definition \ref{def:Kronecker product} can also be equal to 1. Thus, the definition also includes the standard Kronecker products of vectors and matrices.

\begin{remark}
 The Kronecker product is often denoted by the symbol $\otimes$. In order to avoid confusion, we use two different symbols for the tensor product ($\otimes$) and the Kronecker product ($\boxtimes$). 
\end{remark}

\begin{example}
 As simple examples for tensor and Kronecker products, consider the tensors $\mathbf{T}$ and $\mathbf{U}$ with
 \begin{equation*}
  \mathbf{T} = \begin{pmatrix}
                \mathbf{T}_{1,1} & \mathbf{T}_{1,2}\\
                \mathbf{T}_{2,1} & \mathbf{T}_{2,2}
               \end{pmatrix}
  \quad \textrm{and} \quad
  \mathbf{U} = \begin{pmatrix}
                \mathbf{U}_{1,1} \\
                \mathbf{U}_{2,1} \\
                \mathbf{U}_{3,1}
               \end{pmatrix}.
 \end{equation*}
 The respective products are then given by
 \begin{equation*}
  \mathbf{T} \otimes \mathbf{U} = \left(\begin{matrix}
				    \mathbf{T}_{1,1} \mathbf{U}_{1,1} & \mathbf{T}_{1,2} \mathbf{U}_{1,1}\\
				    \mathbf{T}_{2,1} \mathbf{U}_{1,1} & \mathbf{T}_{2,2} \mathbf{U}_{1,1}
				  \end{matrix} \,\middle|\,
				  \begin{matrix}
				    \mathbf{T}_{1,1} \mathbf{U}_{2,1} & \mathbf{T}_{1,2} \mathbf{U}_{2,1}\\
				    \mathbf{T}_{2,1} \mathbf{U}_{2,1} & \mathbf{T}_{2,2} \mathbf{U}_{2,1}
				  \end{matrix} \,\middle|\,
				  \begin{matrix}
				    \mathbf{T}_{1,1} \mathbf{U}_{3,1} & \mathbf{T}_{1,2} \mathbf{U}_{3,1}\\
				    \mathbf{T}_{2,1} \mathbf{U}_{3,1} & \mathbf{T}_{2,2} \mathbf{U}_{3,1}
				  \end{matrix}\right),
 \end{equation*}
 and
 \begin{equation*}
  \mathbf{T} \boxtimes \mathbf{U} = \begin{pmatrix}
                                     \mathbf{T}_{1,1} \mathbf{U}_{1,1} & \mathbf{T}_{1,2} \mathbf{U}_{1,1} \\
                                     \mathbf{T}_{1,1} \mathbf{U}_{2,1} & \mathbf{T}_{1,2} \mathbf{U}_{2,1} \\
                                     \mathbf{T}_{1,1} \mathbf{U}_{3,1} & \mathbf{T}_{1,2} \mathbf{U}_{3,1} \\
                                     \mathbf{T}_{2,1} \mathbf{U}_{1,1} & \mathbf{T}_{2,2} \mathbf{U}_{1,1} \\
                                     \mathbf{T}_{2,1} \mathbf{U}_{2,1} & \mathbf{T}_{2,2} \mathbf{U}_{2,1} \\
                                     \mathbf{T}_{2,1} \mathbf{U}_{3,1} & \mathbf{T}_{2,2} \mathbf{U}_{3,1} 
                                    \end{pmatrix}. \tag*{\exampleSymbol}
 \end{equation*}
\end{example}

\subsection{Tensor-train format}
\label{sec:Tensor-Train Format}

A frequently used tensor format is the so-called \emph{tensor-train format}, or short TT format, see \cite{OSELEDETS2011}. It is a special case of the \emph{hierarchical Tucker format} that has been investigated, e.g.~in \cite{ARNOLD2013,LUBICH2013}. The TT format is one of the most promising tensor formats in terms of storage consumption as well as computational robustness and has been successfully applied to many different application areas, e.g.~quantum physics \cite{WHITE1992,MEYER2009}, chemical reaction dynamics \cite{DOLGOV2015,GELSS2016}, stochastic queuing problems \cite{GELSS2017, KRESSNER2014}, machine learning \cite{NOVIKOV2015,COHEN2015}, and high-dimensional data analysis \cite{KLUS2016,KLUS2016b}. Since we are only interested in exploiting the TT format for the construction of fractal structures, we will disregard numerical aspects concerning memory consumption and the computational costs. For more information on the TT format, we refer to \cite{HACKBUSCH2012}.

\begin{definition}
A tensor $\mathbf{T} \in \mathbb{R}^{n_1 \times \dots \times n_d}$ is said to be in the \emph{TT format} if
\begin{equation*} 
    \mathbf{T} = \sum_{k_0=1}^{r_0} \cdots  \sum_{k_d=1}^{r_d} \left( \mathbf{T}^{(1)} \right)_{k_0,:,k_1} \otimes \dots \otimes \left( \mathbf{T}^{(d)} \right)_{k_{d-1},:,k_d},
\end{equation*}
where the $\mathbf{T}^{(i)} \in \mathbb{R}^{r_{i-1} \times n_i \times r_i}$, $i=1, \dots, d$, are called \emph{TT cores} and the numbers $r_i$ \emph{TT ranks} of the tensor. Here, $ r_0 = r_d = 1 $.
\end{definition}

For the sake of comprehensibility, we represent the TT cores as two-dimensional arrays containing vectors as elements, cf.~\cite{KAZEEV2012}. For a given tensor-train $\mathbf{T} \in \mathbb{R}^{ n_1 \times \dots \times n_d}$ with cores $\mathbf{T}^{(i)} \in \mathbb{R}^{r_{i-1} \times n_i \times r_i}$, $i = 1, \dots, d$, each core is written as
\begin{equation} \label{eq: core notation - single core}
    \left[ \mathbf{T}^{(i)} \right] =
    \begin{bmatrix}
        & \mathbf{T}^{(i)}_{1,:,1} & \cdots & \mathbf{T}^{(i)}_{1,:,r_i} & \\
        & & & & \\
        & \vdots & \ddots & \vdots & \\
        & & & & \\
        & \mathbf{T}^{(i)}_{r_{i-1},:,1} & \cdots & \mathbf{T}^{(i)}_{r_{i-1},:,r_i} &
    \end{bmatrix}.
\end{equation}
We then use the following notation for the TT decomposition of $\mathbf{T}$:
\begin{equation} \label{eq: core notation - all cores}
\begin{split}
    \mathbf{T} & = \left[ \mathbf{T}^{(1)}\right] \otimes \left[ \mathbf{T}^{(2)}\right] \otimes \dots \otimes \left[ \mathbf{T}^{(d-1)}\right] \otimes \left[ \mathbf{T}^{(d)}\right] \\[1.5ex]
    &= 
    \begin{bmatrix}
        \mathbf{T}^{(1)}_{1,:,1} & \cdots & \mathbf{T}^{(1)}_{1,:,r_1}
    \end{bmatrix}
    \otimes 
    \begin{bmatrix}
        \mathbf{T}^{(2)}_{1,:,1}   & \cdots & \mathbf{T}^{(2)}_{1,:,r_2}  \\
        \vdots                       & \ddots & \vdots                        \\
        \mathbf{T}^{(2)}_{r_1,:,1} & \cdots & \mathbf{T}^{(2)}_{r_1,:,r_2}
    \end{bmatrix}
    \otimes \cdots \\[1.5ex]
    & \qquad \cdots \otimes
    \begin{bmatrix}
        \mathbf{T}^{(d-1)}_{1,:,1}       & \cdots & \mathbf{T}^{(d-1)}_{1,:,r_{d-1}}      \\
        \vdots                             & \ddots & \vdots                                  \\
        \mathbf{T}^{(d-1)}_{r_{d-2},:,1} & \cdots & \mathbf{T}^{(d-1)}_{r_{d-2},:,r_{d-1}}
    \end{bmatrix}
    \otimes
    \begin{bmatrix}
        \mathbf{T}^{(d)}_{1,:,1}       \\
        \vdots                           \\
        \mathbf{T}^{(d)}_{r_{d-1},:,1}
    \end{bmatrix}.
    \end{split}
\end{equation}
This operation can be regarded as a generalization of the standard matrix multiplication, where the cores contain vectors as elements instead of scalar values. Just like multiplying two matrices, we compute the tensor products of the corresponding elements and then sum over the columns and rows, respectively. Below, we will use this notation to derive compact representations of fractal patterns in arbitrary dimensions.

\begin{example}
  Let us consider a simple example for a TT decomposition which is similar to the decompositions used later for the fractal construction. We define a tensor $\mathbf{T}$ of order 3 in the TT format as
  \begin{equation*} 
    \mathbf{T} =  \begin{bmatrix}
		    \begin{pmatrix} 1 \\ 0 \\ 1 \end{pmatrix} & \begin{pmatrix} 0 \\ 1 \\ 0 \end{pmatrix}
		  \end{bmatrix} \otimes
		  \begin{bmatrix}
		    \begin{pmatrix} 1 \\ 0 \\ 1 \end{pmatrix} & \begin{pmatrix} 0 \\ 0 \\ 0 \end{pmatrix} \\[0.5cm]
		    \begin{pmatrix} 0 \\ 0 \\ 0 \end{pmatrix} & \begin{pmatrix} 0 \\ 1 \\ 0 \end{pmatrix}
		  \end{bmatrix} \otimes
		  \begin{bmatrix}
		    \begin{pmatrix} 1 \\ 0 \\ 1 \end{pmatrix} \\[0.5cm]
		    \begin{pmatrix} 0 \\ 1 \\ 0 \end{pmatrix}
		  \end{bmatrix}.
  \end{equation*}
  By consecutively contracting the given TT cores, we then obtain
  \begin{equation*} 
    \mathbf{T} =  \begin{bmatrix}
		    \begin{pmatrix} 1 & 0 & 1 \\ 0 & 0 & 0 \\ 1 & 0 & 1 \end{pmatrix} & \begin{pmatrix} 0 & 0 & 0 \\ 0 & 1 & 0 \\ 0 & 0 & 0 \end{pmatrix}
		  \end{bmatrix} \otimes
		  \begin{bmatrix}
		    \begin{pmatrix} 1 \\ 0 \\ 1 \end{pmatrix} \\[0.5cm]
		    \begin{pmatrix} 0 \\ 1 \\ 0 \end{pmatrix}
		  \end{bmatrix} 
	       =  \left(\,\begin{matrix}
		   1 & 0 & 1 \\ 0 & 0 & 0 \\ 1 & 0 & 1
		  \end{matrix}\,\middle|\,
		  \begin{matrix}
		   0 & 0 & 0 \\ 0 & 1 & 0 \\ 0 & 0 & 0
		  \end{matrix}\,\middle|\,
		  \begin{matrix}
		    1 & 0 & 1 \\ 0 & 0 & 0 \\ 1 & 0 & 1
		  \end{matrix}\,\right). \tag*{\exampleSymbol}
  \end{equation*}
\end{example}

\section{Tensor-generated fractals}
\label{sec:Tensor-Generated Fractals}

In this section, we will describe the construction of fractal patterns in different dimensions by using tensor as well as Kronecker products. The idea to exploit the Kronecker product for creating two-dimensional fractals was already described in, e.g., \cite{XUE1996, HARDY2008, LESKOVEC2010}. Here, we will extend this idea in order to create self-similar structures also in higher dimensions. We express these structures by binary tensors $\mathbf{T}$ that only contain the values $0$ and $1$. The created patterns are then visualized by the non-zero entries, i.e.~every entry of $\mathbf{T}$ equal to $1$ is displayed by a line/square/cube whereas the entries equal to $0$ are not displayed. The patterns projected to $[0,1]^{\times d}$ then represent our compact metric spaces. 

Let us again consider the Cantor set. As we saw in the previous sections, it can be constructed by using a finite subdivision rule or an iterated function system. Additionally, this fractal can be represented in terms of Kronecker products. We can represent the iteration steps for the construction by Kronecker powers of the vector $( 1 ~ 0 ~ 1 )$. That is, denoting the $k$th iteration step of the Cantor set by $\mathcal{C}^1_k$, we can write
\begin{equation*}
  \mathcal{C}^1_k = \begin{pmatrix} 1 & 0 & 1 \end{pmatrix}^{\boxtimes k} = \underbrace{\begin{pmatrix} 1 & 0 & 1 \end{pmatrix} \boxtimes \begin{pmatrix} 1 & 0 & 1 \end{pmatrix} \boxtimes \dots \boxtimes \begin{pmatrix} 1 & 0 & 1 \end{pmatrix}}_{k\textrm{ times}}.
\end{equation*}
The iteration steps shown in Figure \ref{fig: Cantor} then correspond to
\begin{center}
\begin{tabular}{lcc}
 $\begin{pmatrix} 1 & 0 & 1 \end{pmatrix}^{\boxtimes 0}$ & $=$ & $\left( \arraycolsep=2pt \begin{array}{@{}*{27}{c}@{}} \phantom{0} & \phantom{0} & \phantom{0} & \phantom{0} & \phantom{0} & \phantom{0} & \phantom{0} & \phantom{0} & \phantom{0} & \phantom{0} & \phantom{0} & \phantom{0} & \phantom{0} & 1 & \phantom{0} & \phantom{0} & \phantom{0} & \phantom{0} & \phantom{0} & \phantom{0} & \phantom{0} & \phantom{0} & \phantom{0} & \phantom{0} & \phantom{0} & \phantom{0} & \phantom{0} \end{array} \right)$, \\[0.1cm]
 $\begin{pmatrix} 1 & 0 & 1 \end{pmatrix}^{\boxtimes 1}$ & $=$ & $\left( \arraycolsep=2pt \begin{array}{@{}*{27}{c}@{}} \phantom{0} & \phantom{0} & \phantom{0} & \phantom{0} & 1 & \phantom{0} & \phantom{0} & \phantom{0} & \phantom{0} & \phantom{0} & \phantom{0} & \phantom{0} & \phantom{0} & 0 & \phantom{0} & \phantom{0} & \phantom{0} & \phantom{0} & \phantom{0} & \phantom{0} & \phantom{0} & \phantom{0} & 1 & \phantom{0} & \phantom{0} & \phantom{0} & \phantom{0} \end{array} \right)$, \\[0.1cm]
 $\begin{pmatrix} 1 & 0 & 1 \end{pmatrix}^{\boxtimes 2}$ & $=$ & $\left( \arraycolsep=2pt \begin{array}{@{}*{27}{c}@{}} \phantom{0} & 1 & \phantom{0} & \phantom{0} & 0 & \phantom{0} & \phantom{0} & 1 & \phantom{0} & \phantom{0} & 0 & \phantom{0} & \phantom{0} & 0 & \phantom{0} & \phantom{0} & 0 & \phantom{0} & \phantom{0} & 1 & \phantom{0} & \phantom{0} & 0 & \phantom{0} & \phantom{0} & 1 & \phantom{0} \end{array} \right)$, \\[0.1cm]
 $\begin{pmatrix} 1 & 0 & 1 \end{pmatrix}^{\boxtimes 3}$ & $=$ & $\left( \arraycolsep=2pt \begin{array}{@{}*{27}{c}@{}} 1 & 0 & 1 & 0 & 0 & 0 & 1 & 0 & 1 & 0 & 0 & 0 & 0 & 0 & 0 & 0 & 0 & 0 & 1 & 0 & 1 & 0 & 0 & 0 & 1 & 0 & 1 \end{array} \right)$. \\
\end{tabular}
\end{center}

In a similar way, we can express the iteration steps of the construction of the Sierpinski carpet and the Menger sponge by (generalized) Kronecker products. If we denote the construction steps of the Sierpinski carpet by $\mathcal{C}^2_k$ and the construction steps of the Menger sponge as $\mathcal{C}^3_k$, these patterns can be expressed as
\begin{equation}\label{eq: sierpinksi, menger}
  \mathcal{C}^2_k =  \begin{pmatrix} 1 & 1 & 1 \\ 1 & 0 & 1 \\ 1 & 1 & 1 \end{pmatrix}^{\boxtimes k} 
  \quad \textrm{and} \quad 
  \mathcal{C}^3_k =  \left(\,\begin{matrix} 1 & 1 & 1 \\ 1 & 0 & 1 \\ 1 & 1 & 1 \end{matrix} \,\middle|\, \begin{matrix} 1 & 0 & 1 \\ 0 & 0 & 0 \\ 1 & 0 & 1 \end{matrix} \,\middle|\, \begin{matrix}  1 & 1 & 1 \\ 1 & 0 & 1 \\ 1 & 1 & 1 \end{matrix}\,\right)^{\boxtimes k},
\end{equation}  
respectively. For instance, Figure \ref{fig: sierpinski} (c) then corresponds to 
\newcommand\tikzmark[1]{%
  \tikz[remember picture,overlay]\coordinate (#1);}
\begin{equation*}
  \mathcal{C}^2_2 =  \begin{pmatrix} 1 & 1 & 1 \\ 1 & 0 & 1 \\ 1 & 1 & 1 \end{pmatrix}^{\boxtimes 2} =  
  \begin{pmatrix} 
  1 & 1 & 1\tikzmark{a} & 1 & 1 & 1\tikzmark{c} & 1 & 1 & 1 \\ 
  1 & 0 & 1             & 1 & 0 & 1 & 1 & 0 & 1 \\ 
  1 & 1 & 1             & 1 & 1 & 1 & 1 & 1 & 1 \\[0.15cm]
  1\tikzmark{e} & 1 & 1             & 0 & 0 & 0 & 1 & 1 & 1\tikzmark{f} \\ 
  1 & 0 & 1             & 0 & 0 & 0 & 1 & 0 & 1 \\ 
  1 & 1 & 1             & 0 & 0 & 0 & 1 & 1 & 1 \\[0.15cm]
  1\tikzmark{g} & 1 & 1             & 1 & 1 & 1 & 1 & 1 & 1\tikzmark{h} \\ 
  1 & 0 & 1             & 1 & 0 & 1 & 1 & 0 & 1 \\ 
  1 & 1 & 1\tikzmark{b} & 1 & 1 & 1\tikzmark{d} & 1 & 1 & 1 
  \end{pmatrix}.
\end{equation*} 
\tikz[remember picture,overlay]{
  \draw[dotted] ([xshift=0.8\tabcolsep,yshift=7.8pt]a.north) -- ([xshift=0.8\tabcolsep,yshift=-0.9pt]b.south);
  \draw[dotted] ([xshift=0.8\tabcolsep,yshift=7.8pt]c.north) -- ([xshift=0.8\tabcolsep,yshift=-0.9pt]d.south);
  \draw[dotted] ([xshift=-0.2cm,yshift=11.9pt]e.north) -- ([xshift=0.04cm,yshift=11.9pt]f.south);
  \draw[dotted] ([xshift=-0.2cm,yshift=11.9pt]g.north) -- ([xshift=0.04cm,yshift=11.9pt]h.south);
}

Note that we can directly deduce the fractal dimensions of these patterns by counting the non-zero entries of the defining tensors. That is, for a defining tensor $\mathbf{T} \in \mathbb{R}^{n \times n \times \dots \times n}$ it holds that
\begin{equation}\label{eq: fractal dimension - tensor}
  D_F = \frac{\ln(m)}{\ln(n)},
\end{equation}  
where $m$ is the number of entries of $\mathbf{T}$ with $\mathbf{T}_{x_1, \dots , x_d} = 1$. Considering Definition \ref{def: fractal dimension}, the number $m$ corresponds to the number of boxes required to cover the pattern. Since we assume that all dimensions of the defining tensor are equal to $n$, the side length of these boxes is $1/n$.

Two further examples of three-dimensional fractals are the Cantor dust, an alternative generalization of the Cantor set, and the 3D Vicsek fractal, see \cite{VICSEK1992}. Figure \ref{fig: dust, vicsek} illustrates the first construction steps of these structures. The defining tensors are given by
\begin{equation*}
  \left(\,\begin{matrix} 1 & 0 & 1 \\ 0 & 0 & 0 \\ 1 & 0 & 1 \end{matrix} \,\middle|\, \begin{matrix} 0 & 0 & 0 \\ 0 & 0 & 0 \\ 0 & 0 & 0 \end{matrix} \,\middle|\, \begin{matrix}  1 & 0 & 1 \\ 0 & 0 & 0 \\ 1 & 0 & 1 \end{matrix}\,\right)
\end{equation*}  
for the Cantor dust, and 
\begin{equation*}
  \left(\,\begin{matrix} 0 & 0 & 0 \\ 0 & 1 & 0 \\ 0 & 0 & 0 \end{matrix} \,\middle|\, \begin{matrix} 0 & 1 & 0 \\ 1 & 1 & 1 \\ 0 & 1 & 0 \end{matrix} \,\middle|\, \begin{matrix}  0 & 0 & 0 \\ 0 & 1 & 0 \\ 0 & 0 & 0 \end{matrix}\,\right) 
\end{equation*}  
for the Vicsek fractal. Similar to the Cantor set, the Cantor dust is totally disconnected. Therefore, it has a topological dimension of 0. The construction of the Vicsek fractal as shown in Figure \ref{fig: dust, vicsek} (d)--(f) converges to a curve within the unit cube, which implies that the topological dimension of this fractal is 1. Following \eqref{eq: fractal dimension - tensor}, we deduce the fractal dimensions $D_F= \ln(8)/\ln(3) \approx 1.8928$ and $D_F=\ln(7)/\ln(3) \approx 1.7712$, respectively.  

\begin{figure}[htb]
    \centering
    \begin{subfigure}[b]{0.3\textwidth}
        \centering
        \includegraphics[width=80px]{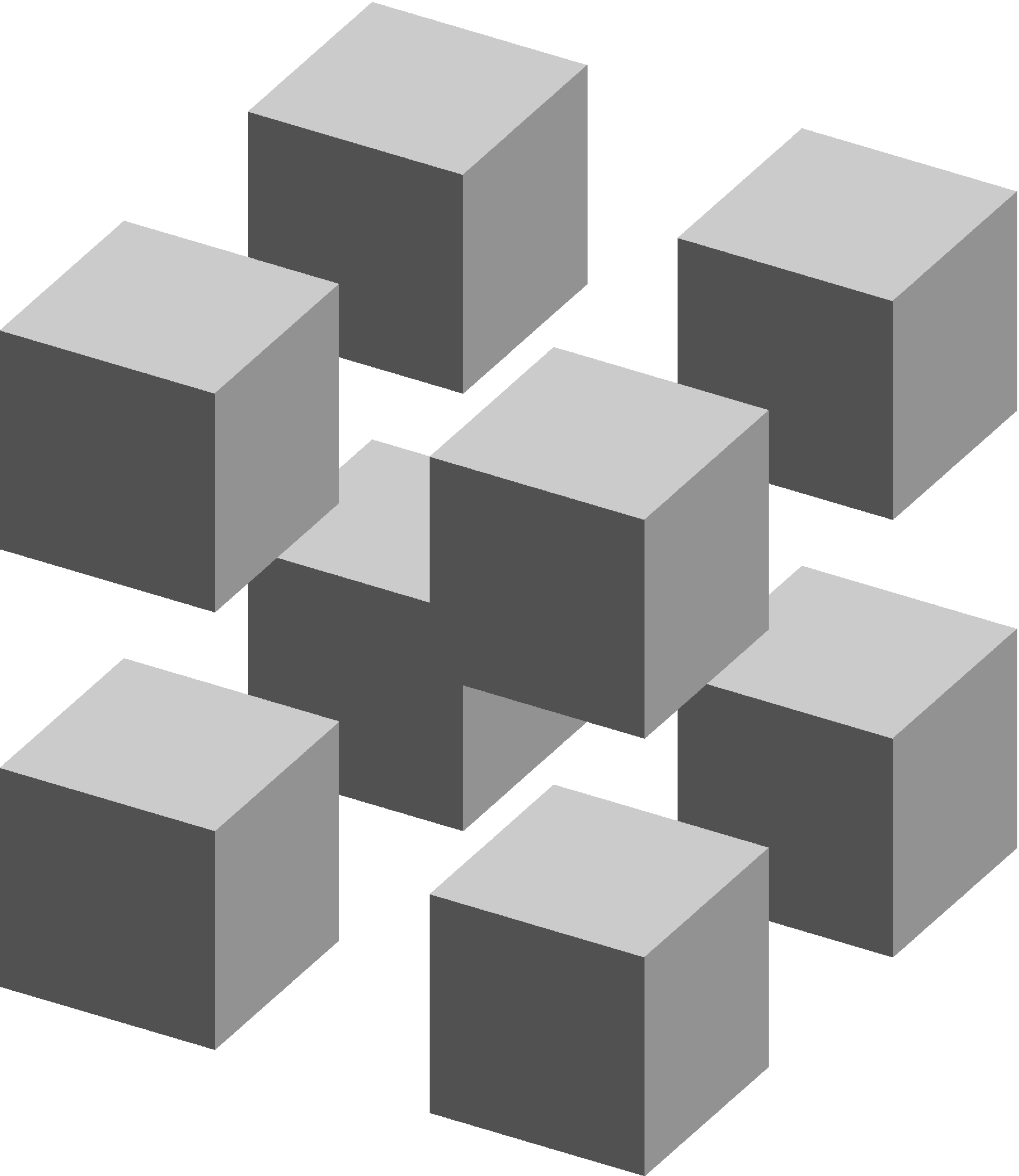}
        \caption{}
    \end{subfigure}
    \hfill
    \begin{subfigure}[b]{0.3\textwidth}
        \centering
        \includegraphics[width=80px]{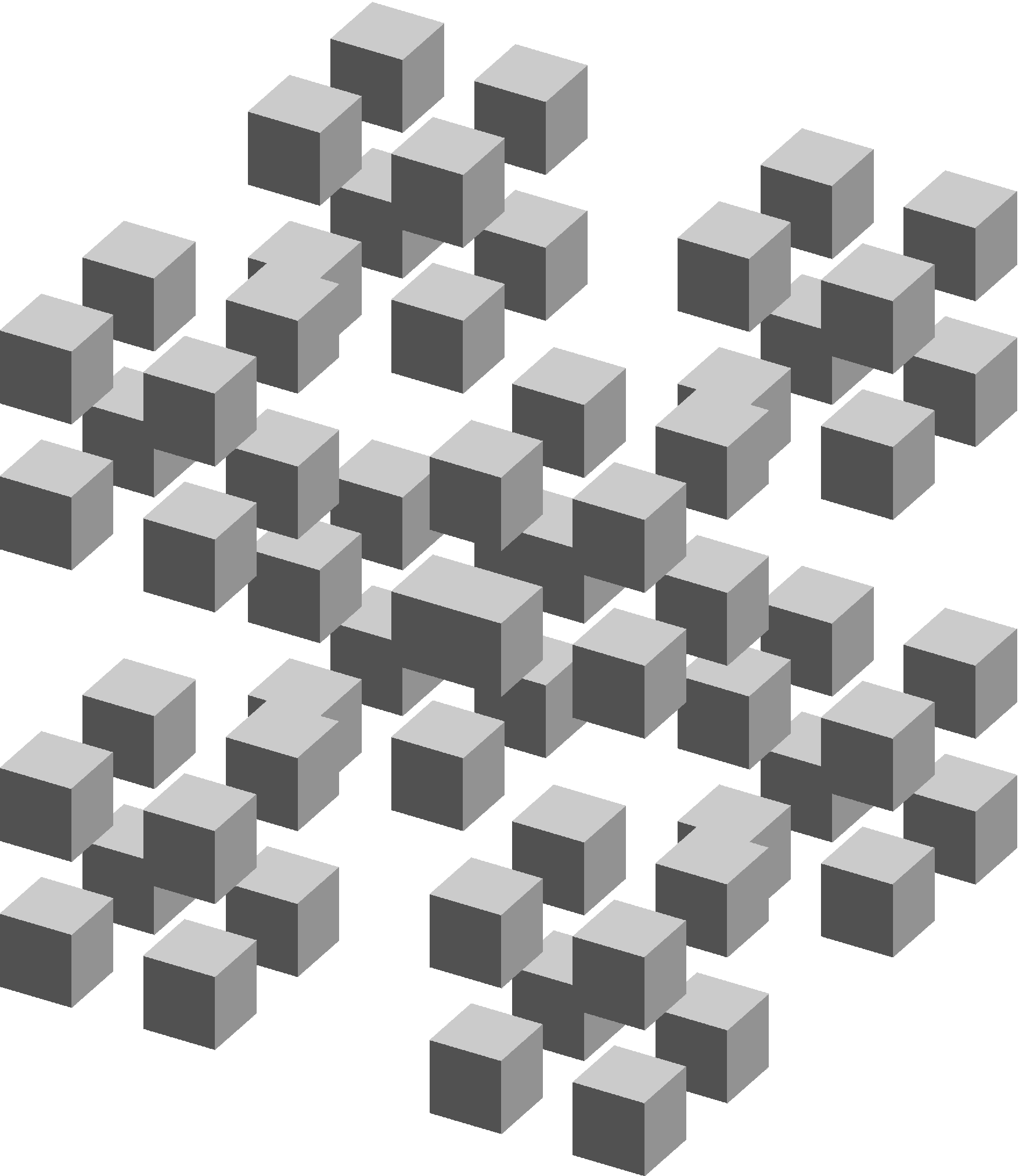}
        \caption{}
    \end{subfigure}
    \hfill
    \begin{subfigure}[b]{0.3\textwidth}
        \centering
        \includegraphics[width=80px]{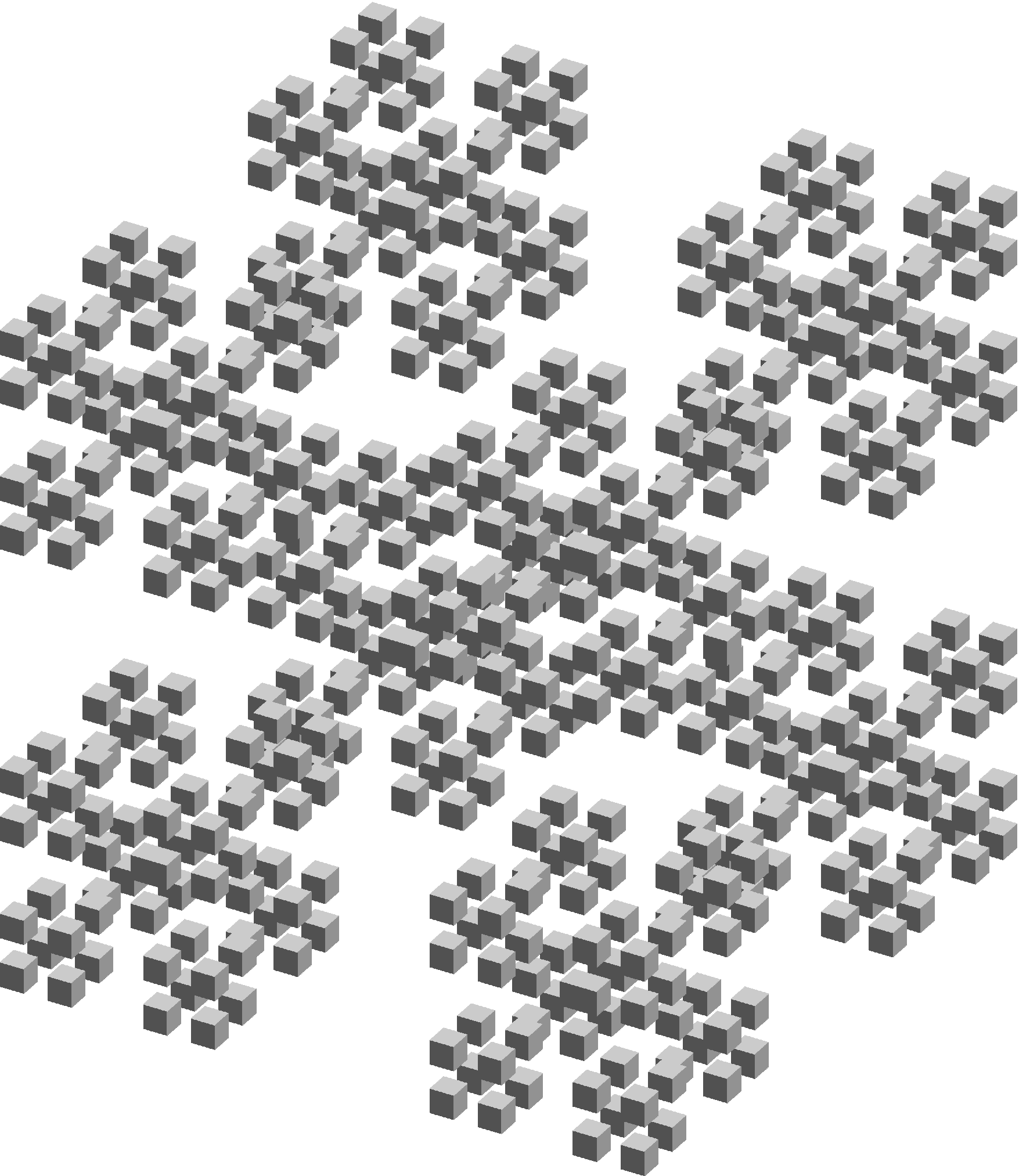}
        \caption{}
    \end{subfigure}\\[0.2cm]
    \begin{subfigure}[b]{0.3\textwidth}
        \centering
        \includegraphics[width=80px]{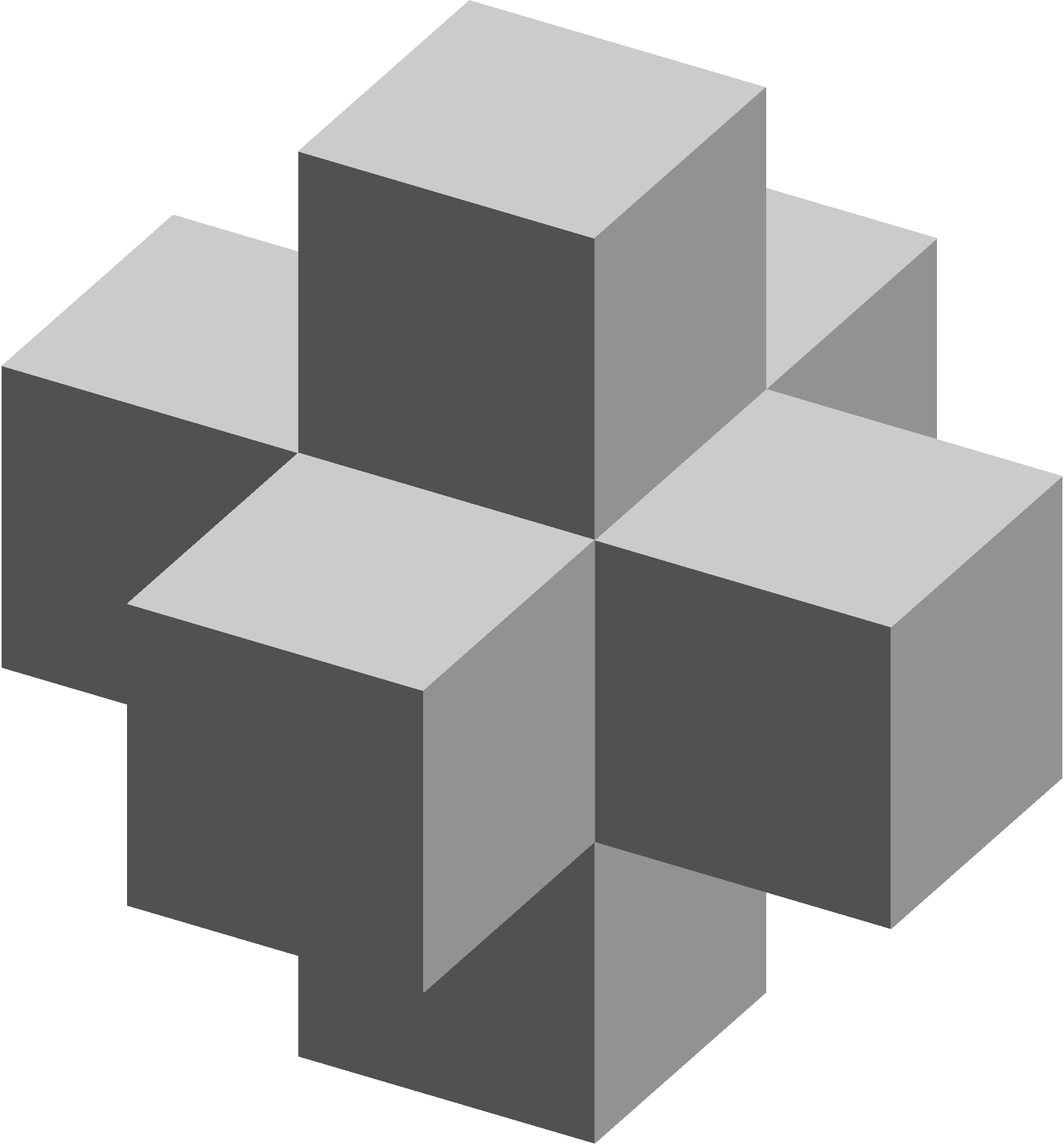}
        \caption{}
    \end{subfigure}
    \hfill
    \begin{subfigure}[b]{0.3\textwidth}
        \centering
        \includegraphics[width=80px]{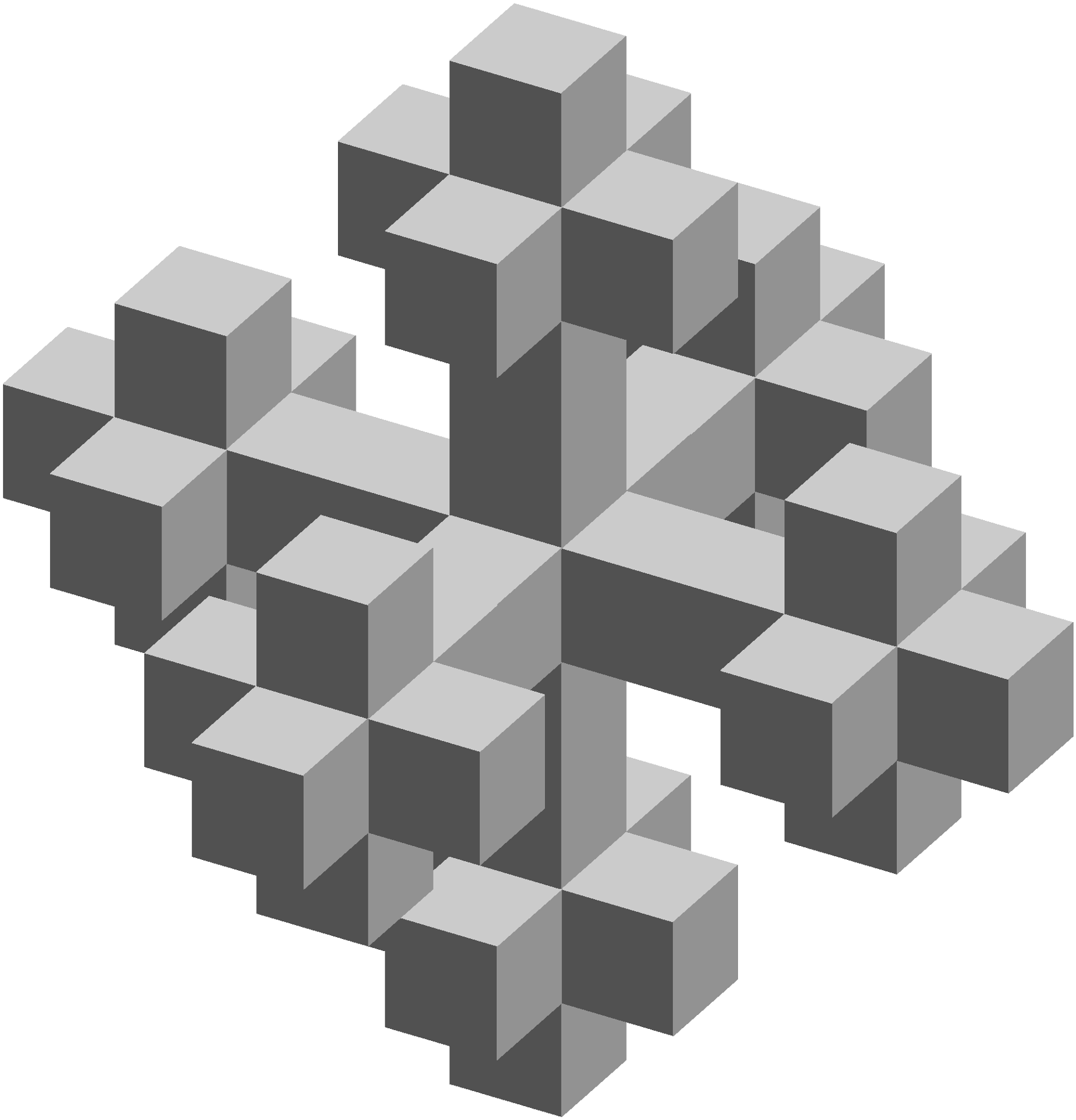}
        \caption{}
    \end{subfigure}
    \hfill
    \begin{subfigure}[b]{0.3\textwidth}
        \centering
        \includegraphics[width=80px]{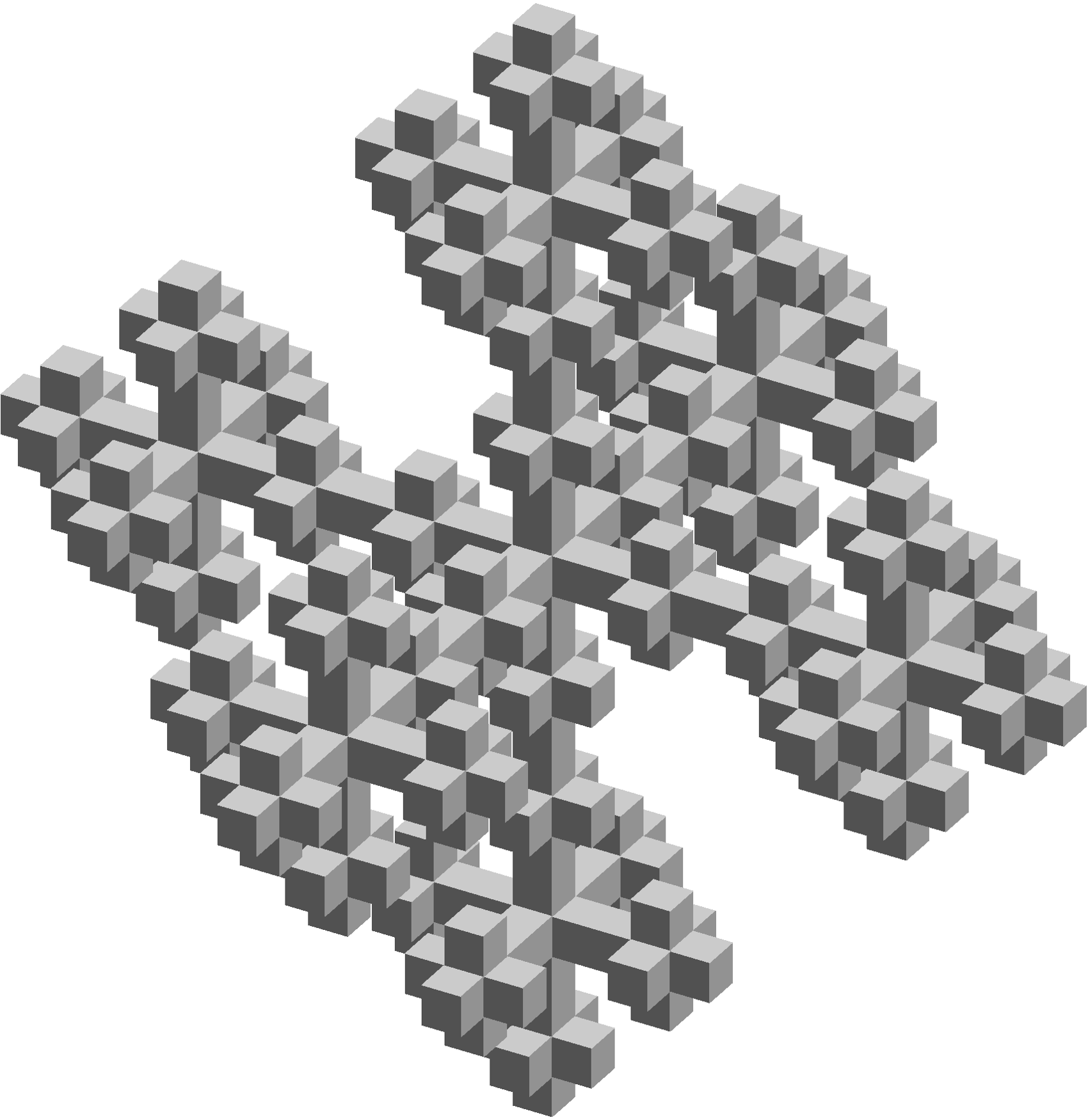}
        \caption{}
    \end{subfigure}
    \caption{Construction of the Cantor dust and Vicsek fractal: The first three iteration steps of the tensor-based construction of the three-dimensional Cantor dust are shown in (a)--(c), the iteration steps of the construction of the three-dimensional Vicsek fractal in (d)--(f). }
    \label{fig: dust, vicsek}
\end{figure}

In order to illustrate the multidimensional fractal construction using tensor-train decompositions, let us consider again the Sierpinski carpet, see Figure \ref{fig: sierpinski}, and the Menger sponge, see Figure \ref{fig: menger}. By decomposing the defining tensors given in \eqref{eq: sierpinksi, menger} into the TT format, we obtain
\begin{equation}\label{eq: TT 1}
\mathcal{C}^2_k =  \left( \begin{bmatrix} \begin{pmatrix} 1 \\ 1 \\ 1 \end{pmatrix} & \begin{pmatrix} 1 \\ 0 \\ 1 \end{pmatrix} \end{bmatrix} \otimes \begin{bmatrix} \begin{pmatrix} 1 \\ 0 \\ 1 \end{pmatrix} \\[0.5cm] \begin{pmatrix} 0 \\ 1 \\ 0 \end{pmatrix} \end{bmatrix} \right)^{\boxtimes k} ,
\end{equation}  
and
\begin{equation}\label{eq: TT 2}
\mathcal{C}^3_k =  \left( \begin{bmatrix} \begin{pmatrix} 1 \\ 1 \\ 1 \end{pmatrix} & \begin{pmatrix} 1 \\ 0 \\ 1 \end{pmatrix} \end{bmatrix} \otimes \begin{bmatrix} \begin{pmatrix} 1 \\ 0 \\ 1 \end{pmatrix} & \begin{pmatrix} 0 \\ 0 \\ 0 \end{pmatrix} \\[0.5cm] \begin{pmatrix} 0 \\ 1 \\ 0 \end{pmatrix} & \begin{pmatrix} 1 \\ 0 \\ 1 \end{pmatrix} \end{bmatrix} \otimes \begin{bmatrix} \begin{pmatrix} 1 \\ 0 \\ 1 \end{pmatrix} \\[0.5cm] \begin{pmatrix} 0 \\ 1 \\ 0 \end{pmatrix} \end{bmatrix} \right)^{\boxtimes k} ,
\end{equation} 
respectively. Here, we use the notation introduced at the end of Section \ref{sec:Tensor-Train Format}. One can see that the first and the last cores, respectively, of both decompositions \eqref{eq: TT 1} and \eqref{eq: TT 2} are equal. The question now is: What happens if we insert more middle cores? We fix the first and the last core of \eqref{eq: TT 2} and alter the number of cores in between, i.e.~the construction steps of the $d$-dimensional pattern can be expressed as
\begin{equation}\label{eq: TT 3}
\mathcal{C}^d_k =  \left( \begin{bmatrix} \begin{pmatrix} 1 \\ 1 \\ 1 \end{pmatrix} & \begin{pmatrix} 1 \\ 0 \\ 1 \end{pmatrix} \end{bmatrix} \otimes \underbrace{\begin{bmatrix} \begin{pmatrix} 1 \\ 0 \\ 1 \end{pmatrix} & \begin{pmatrix} 0 \\ 0 \\ 0 \end{pmatrix} \\[0.5cm] \begin{pmatrix} 0 \\ 1 \\ 0 \end{pmatrix} & \begin{pmatrix} 1 \\ 0 \\ 1 \end{pmatrix} \end{bmatrix} \otimes \dots \otimes \begin{bmatrix} \begin{pmatrix} 1 \\ 0 \\ 1 \end{pmatrix} & \begin{pmatrix} 0 \\ 0 \\ 0 \end{pmatrix} \\[0.5cm] \begin{pmatrix} 0 \\ 1 \\ 0 \end{pmatrix} & \begin{pmatrix} 1 \\ 0 \\ 1 \end{pmatrix} \end{bmatrix}}_{d-2\textrm{ times}} \otimes \begin{bmatrix} \begin{pmatrix} 1 \\ 0 \\ 1 \end{pmatrix} \\[0.5cm] \begin{pmatrix} 0 \\ 1 \\ 0 \end{pmatrix} \end{bmatrix} \right)^{\boxtimes k} ,
\end{equation} 
for $d \geq 2$. This systematic approach resembles the concept of so-called SLIM decompositions, which we introduced in \cite{GELSS2017}. It was shown that TT decompositions similar to \eqref{eq: TT 3} can be used to represent high-dimensional nearest-neighbor interaction systems. However, we here adapt the idea of SLIM decompositions for creating multidimensional counterparts of the Menger sponge, so-called \emph{multisponges}.

In order to calculate the fractal dimension $D_F^d$ of a multisponge in $\mathbb{R}^d$, we again consider the number $m$ of non-zero entries of the defining tensor. Due to the construction of the TT decompositions in \eqref{eq: TT 3}, $m$ can be computed in terms of the non-zero entries of every core vector resulting in
\begin{equation*}
m = (d+2) \cdot 2^{d-1}.
\end{equation*}  
A proof of this result can be found in Appendix \ref{app: non-zero multisponge}. The scaling factor is again $1/3$ so that the fractal dimension is given by
\begin{equation*}
D_F^d  = \frac{\ln(d+2) + (d-1)\cdot\ln(2)}{\ln(3)}.
\end{equation*}  

The $k$th iteration step of the $d$-dimensional multisponge comprises $m^k$ smaller multisponges, each with side lengths of $\frac{1}{3}^k$. Thus, the total ($d$-dimensional) volume at each step is
\begin{equation}\label{eq: volume multisponge}
V_k = m^k \cdot \left( \left( \frac{1}{3}\right)^{k} \right)^d = \left( \frac{m}{3^d}\right)^k = \left( \frac{(d+2) \cdot 2^{d-1}}{3^d}\right)^k,
\end{equation} 
which approaches 0 for $k \rightarrow \infty$ (see Appendix \ref{app: volume multisponge} for a proof). For fixed $k$, any subset of $C^d_k$ with dimension larger than $1$ and smaller than $d$ will be thoroughly perforated during the further iterations. This implies that the limit $C^d_\infty$, projected to $[0,1]^{\times d}$, can only have a topological dimension $d_T \leq 1$. Since $C^d_k$ is connected (see Appendix \ref{app: connectedness multisponge} for a proof), the topological dimension is equal to 1. Thus, the $d$-dimensional multisponge is a fractal curve.

At the end of this section, we additionally want to present a different approach for combining Kronecker and tensor products in order to construct fractal patterns. Consider a three-dimensional tensor $\mathbf{T} \in \mathbb{R}^{n \times n \times 3}$ representing RGB triplets. Every vector $\mathbf{T}_{x,y,:}$, $1 \leq x,y \leq n$, defines the color for one pixel $(x,y)$ of an image by specifying the intensities of red, green, and blue components of the color. For three given defining matrices $M_R, M_G, M_B \in \mathbb{R}^{m \times m}$, we compute the Kronecker powers $M_R^{\boxtimes d}$, $M_G^{\boxtimes d}$, $M_B^{\boxtimes d}$ and set $\mathbf{T}_{:,:,1} = M_R^{\boxtimes d}$, $\mathbf{T}_{:,:,2} = M_G^{\boxtimes d}$, and $\mathbf{T}_{:,:,3} = M_B^{\boxtimes d}$. That is, using a similar core notation for the Kronecker product as for the tensor product, cf.~\eqref{eq: core notation - single core} and \eqref{eq: core notation - all cores}, the RGB tensor $\mathbf{T} \in \mathbb{R}^{n \times n \times 3}$, $n = m^d$, can be written as 
\begin{equation*}
\mathbf{T} =   \begin{bmatrix} M_R & M_G & M_B \end{bmatrix} \boxtimes \begin{bmatrix} M_R & 0 & 0 \\ 0 & M_G & 0 \\ 0 & 0 & M_B \end{bmatrix} \boxtimes \dots \boxtimes \begin{bmatrix} M_R & 0 & 0 \\ 0 & M_G & 0 \\ 0 & 0 & M_B \end{bmatrix} \otimes \begin{bmatrix} \begin{pmatrix} 1 \\ 0 \\ 0 \end{pmatrix} \\[0.5cm] \begin{pmatrix} 0 \\ 1 \\ 0 \end{pmatrix} \\[0.5cm] \begin{pmatrix} 0 \\ 0 \\ 1 \end{pmatrix} \end{bmatrix}.
\end{equation*}
In this way, we construct one fractal pattern per RGB layer and combine these layers to a quasi fractal image. Figure \ref{fig: color} shows three examples of RGB fractals. The corresponding matrices $M_R$, $M_G$, and $M_B$ can be found in Appendix \ref{app: color}. Note that the images in Figure \ref{fig: color} as such are not strictly fractal. However, we can still find quasi self-similar patterns at different scales if we ignore the varying intensities of the RGB channels. 

\begin{figure}[htb]
    \centering
    \begin{subfigure}[b]{0.3\textwidth}
        \centering
        \includegraphics[width=100px]{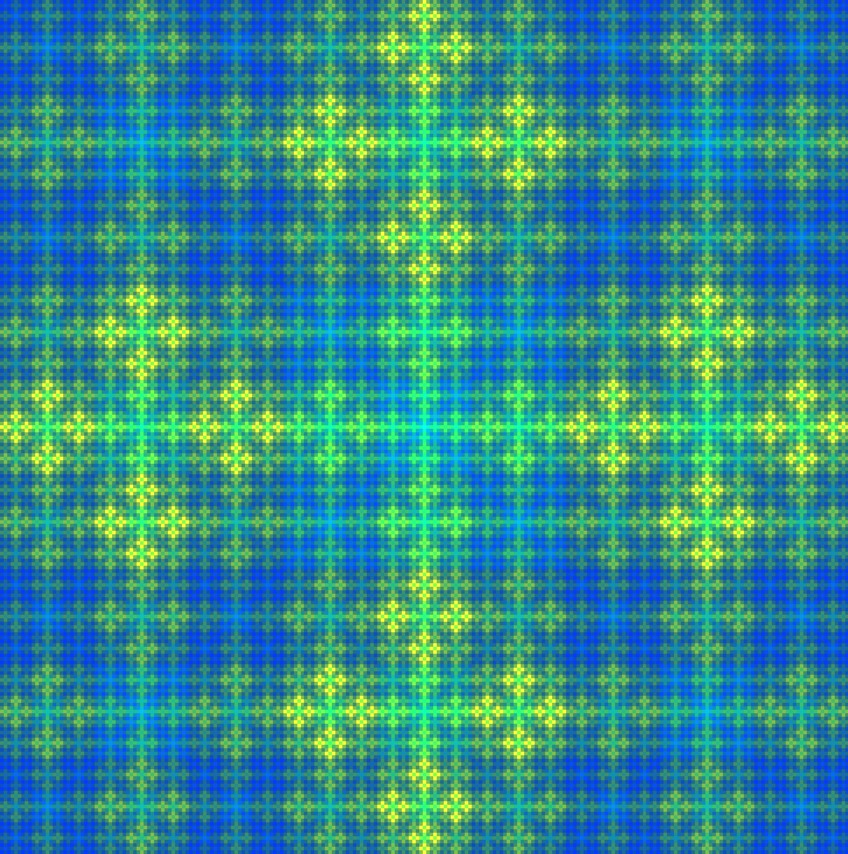}
        \caption{}
        \label{fig: color a}
    \end{subfigure}
    \hfill
    \begin{subfigure}[b]{0.3\textwidth}
        \centering
        \includegraphics[width=100px]{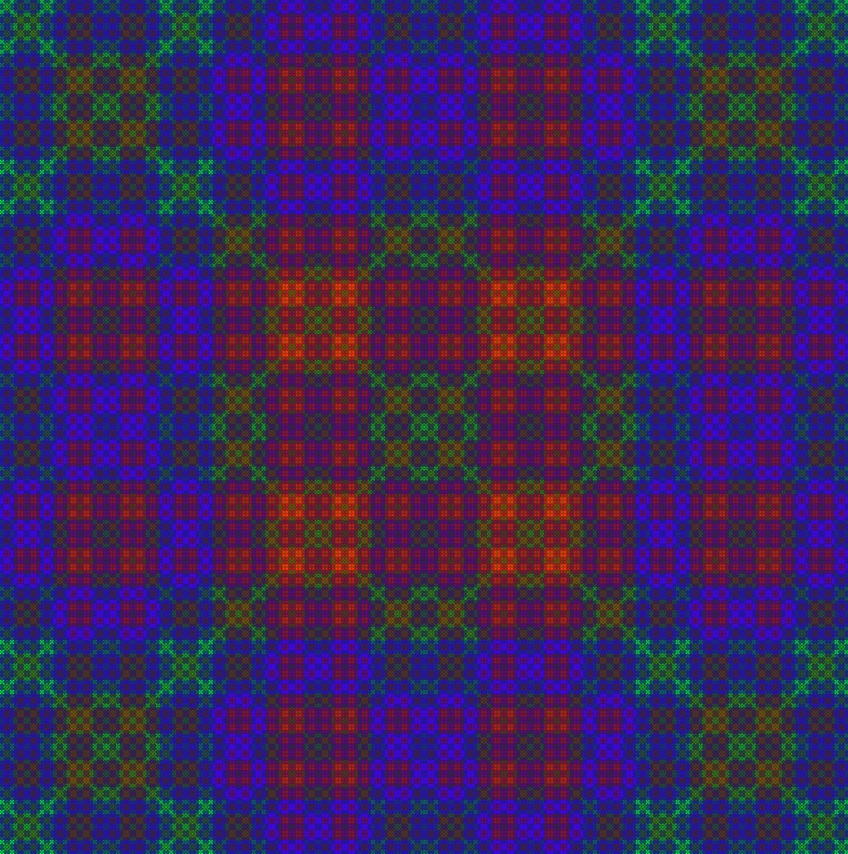}
        \caption{}
        \label{fig: color b}
    \end{subfigure}
    \hfill
    \begin{subfigure}[b]{0.3\textwidth}
        \centering
        \includegraphics[width=100px]{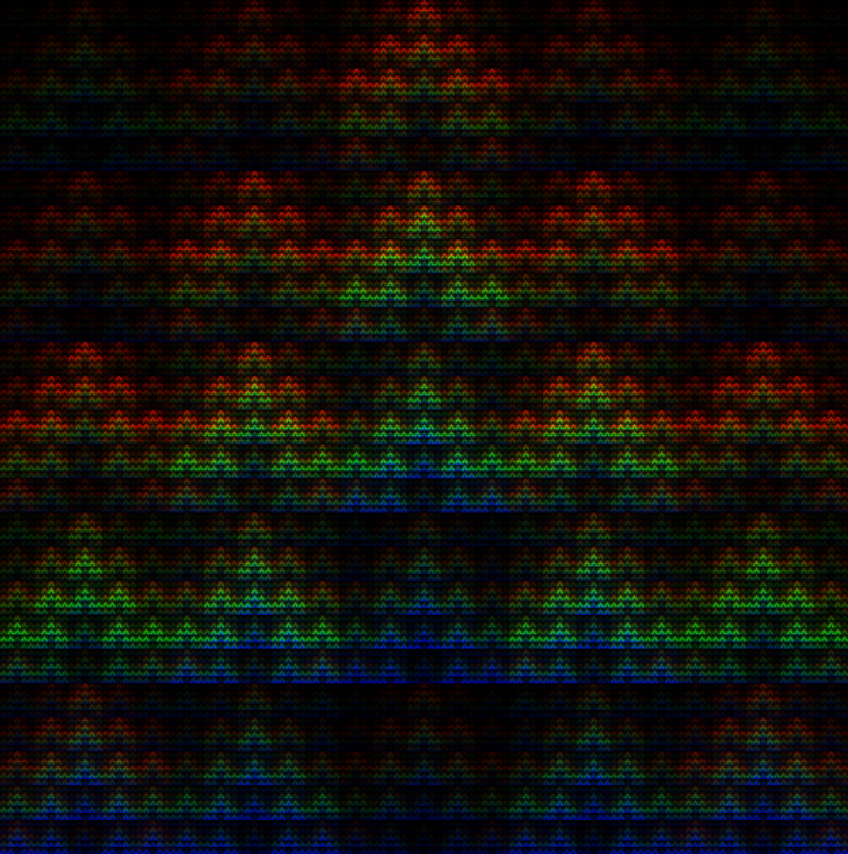}
        \caption{}
        \label{fig: color c}
    \end{subfigure}
    \caption{RGB fractals: The pictures shown are given by three-dimensional tensors $\mathbf{T} \in \mathbb{R}^{n \times n \times 3}$ constructed as described above. The respective mode sizes $n$ are given by $n = m^5$ with (a) $m=3$, (b) $m = 4$, and (c) $m=5$.}
    \label{fig: color}
\end{figure}

\section{Conclusion}
\label{sec:Conclusion}

In this paper, we proposed a method to construct geometric fractals by exploiting Kronecker products and tensor decompositions. In particular, we focused on the so-called TT format which has been frequently used in the past years to mitigate the curse of dimensionality in many different application areas. The approach was illustrated with several examples for well-known fractal patterns which are usually created by using, e.g., iterated function systems. First, we have shown how to construct self-similar structures in one, two, and three dimensions. Then, we described the generalization to arbitrary dimensions by constructing so-called multisponges. Additionally, we presented the application of tensor decompositions and Kronecker products to construct RGB fractals. 

We demonstrated that it is possible to use tensor decompositions for the construction of fractal patterns, which may be one step towards the aim of using tensor decompositions for finding self-similar structures in complex numerical data and providing tensor-based representations of the data with high compression ratios. Other possible applications may be found in the fields of fractal image compression or computer graphics. 

Our future work will include the investigation of the relation between self-similar structures of tensors and the characteristics of the cores of corresponding TT representations. Moreover, we will consider the tensor-based construction of further (quasi) fractal structures as well as the detection of self-similar patterns in large data sets by applying different tensor formats.

\section*{Acknowledgements}

This research has been funded by Deutsche Forschungsgemeinschaft (DFG) through grant CRC 1114 `Scaling Cascades in Complex Systems', Project B03 `Multilevel coarse graining of multiscale problems'. Special thanks should be given to Sebastian Matera and Stefan Klus for the interesting discussions and valuable inputs.

\bibliographystyle{unsrturl}
\bibliography{References}

\begin{thebibliography}{10}

\bibitem{RICHARDSON1961}
L.~F. Richardson.
\newblock {\em The problem of contiguity: An appendix to statistics of deadly
  quarrels}, volume~6, pages 140--187.
\newblock Ann Arbor, 1961.

\bibitem{MANDELBROT1967}
B.~Mandelbrot.
\newblock How long is the coast of {B}ritain? {S}tatistical self-similarity and
  fractional dimension.
\newblock {\em Science}, 156(3775):636--638, 1967.
\newblock \href {http://dx.doi.org/10.1126/science.156.3775.636}
  {\path{doi:10.1126/science.156.3775.636}}.

\bibitem{KOCH1904}
H.~von Koch.
\newblock {\em Sur une courbe continue sans tangente obtenue par une
  construction geometrique elementaire}.
\newblock Norstedt \& soner, 1904.

\bibitem{WEIERSTRASS1886}
K.~Weierstrass.
\newblock {\em {A}bhandlungen aus der {F}unctionenlehre}.
\newblock J. Springer, 1886.

\bibitem{MANDELBROT1975}
B.~Mandelbrot.
\newblock {\em Les objets fractals: Forme, hasard et dimension}.
\newblock Flammarion, 1975.

\bibitem{RINALDO1993}
A.~Rinaldo, I.~Rodriguez-Iturbe, R.~Rigon, E.~Ijjasz-Vasquez, and R.~L. Bras.
\newblock Self-organized fractal river networks.
\newblock {\em Physical Review Letters}, 70:822--825, 1993.
\newblock \href {http://dx.doi.org/10.1103/PhysRevLett.70.822}
  {\path{doi:10.1103/PhysRevLett.70.822}}.

\bibitem{ZEIDE1991}
B.~Zeide and P.~Pfeifer.
\newblock A method for estimation of fractal dimension of tree crowns.
\newblock {\em Forest science}, 37(5):1253--1265, 1991.

\bibitem{PEITGEN2004}
H.~O. Peitgen, H.~J{\"u}rgens, and D.~Saupe.
\newblock {\em Chaos and fractals: {N}ew frontiers of science}.
\newblock Springer New York, 2004.
\newblock \href {http://dx.doi.org/10.1007/b97624} {\path{doi:10.1007/b97624}}.

\bibitem{MANDELBROT1982}
B.~Mandelbrot.
\newblock {\em The fractal geometry of nature}.
\newblock Freeman and Co., 1982.

\bibitem{BOLZANO1930}
B.~Bolzano and K.~Rychl{\'\i}k.
\newblock {\em Functionenlehre}.
\newblock K{\"o}nigliche b{\"o}hmische Gesellschaft der Wissenschaften, 1930.

\bibitem{PEANO1890}
G.~Peano.
\newblock Sur une courbe, qui remplit toute une aire plane.
\newblock {\em Mathematische Annalen}, 36(1):157--160, 1890.
\newblock \href {http://dx.doi.org/10.1007/BF01199438}
  {\path{doi:10.1007/BF01199438}}.

\bibitem{HILBERT1891}
D.~Hilbert.
\newblock {\"U}ber die stetige {A}bbildung einer {L}inie auf ein
  {F}l{\"a}chenst{\"u}ck.
\newblock {\em Mathematische Annalen}, 38(3):459--460, 1891.
\newblock \href {http://dx.doi.org/10.1007/BF01199431}
  {\path{doi:10.1007/BF01199431}}.

\bibitem{WERNER1999}
D.~H. Werner, R.~L. Haupt, and P.~L. Werner.
\newblock Fractal antenna engineering: The theory and design of fractal antenna
  arrays.
\newblock {\em IEEE Antennas and Propagation Magazine}, 41(5):37--58, 1999.
\newblock \href {http://dx.doi.org/10.1109/74.801513}
  {\path{doi:10.1109/74.801513}}.

\bibitem{BARNSLEY1993}
M.~F. Barnsley and L.~P. Hurd.
\newblock {\em Fractal image compression}.
\newblock A. K. Peters, Ltd., 1993.

\bibitem{ENCARNACAO1992}
J.~L. Encarnacao, H.~O. Peitgen, G.~Sakas, and G.~Englert.
\newblock {\em Fractal geometry and computer graphics}.
\newblock {B}eitr{\"a}ge zur {G}raphischen {D}atenverarbeitung. Springer Berlin
  Heidelberg, 1992.
\newblock \href {http://dx.doi.org/10.1007/978-3-642-95678-2}
  {\path{doi:10.1007/978-3-642-95678-2}}.

\bibitem{KAZEEV2014}
V.~Kazeev, M.~Khammash, M.~Nip, and C.~Schwab.
\newblock Direct solution of the chemical master equation using quantized
  tensor trains.
\newblock {\em PLOS Computational Biology}, 10(3):e1003359, 2014.
\newblock \href {http://dx.doi.org/10.1371/journal.pcbi.1003359}
  {\path{doi:10.1371/journal.pcbi.1003359}}.

\bibitem{DOLGOV2015}
S.~Dolgov and B.~Khoromskij.
\newblock Simultaneous state-time approximation of the chemical master equation
  using tensor product formats.
\newblock {\em Numerical Linear Algebra with Applications}, 22(2):197--219,
  2015.
\newblock \href {http://dx.doi.org/10.1002/nla.1942}
  {\path{doi:10.1002/nla.1942}}.

\bibitem{GELSS2017}
P.~Gel{\ss}, S.~Klus, S.~Matera, and C.~Sch{\"u}tte.
\newblock Nearest-neighbor interaction systems in the tensor-train format.
\newblock {\em Journal of Computational Physics}, 341:140--162, 2017.
\newblock \href {http://dx.doi.org/10.1016/j.jcp.2017.04.007}
  {\path{doi:10.1016/j.jcp.2017.04.007}}.

\bibitem{HITCHCOCK1927}
F.~L. Hitchcock.
\newblock The expression of a tensor or a polyadic as a sum of products.
\newblock {\em Journal of Mathematics and Physics}, 6:164--189, 1927.
\newblock \href {http://dx.doi.org/10.1002/sapm192761164}
  {\path{doi:10.1002/sapm192761164}}.

\bibitem{TUCKER1963}
L.~R. Tucker.
\newblock Implications of factor analysis of three-way matrices for measurement
  of change.
\newblock In C.~W. Harris, editor, {\em Problems in measuring change}, pages
  122--137. University of Wisconsin Press, 1963.

\bibitem{TUCKER1964}
L.~R. Tucker.
\newblock {T}he extension of factor analysis to three-dimensional matrices.
\newblock In H.~Gulliksen and N.~Frederiksen, editors, {\em Contributions to
  mathematical psychology}, pages 110--127. Holt, Rinehart and Winston, 1964.

\bibitem{OSELEDETS2009}
I.~V. Oseledets.
\newblock A new tensor decomposition.
\newblock {\em Doklady Mathematics}, 80(1):495--496, 2009.
\newblock \href {http://dx.doi.org/10.1134/S1064562409040115}
  {\path{doi:10.1134/S1064562409040115}}.

\bibitem{OSELEDETS2009b}
I.~V. Oseledets and E.~E. Tyrtyshnikov.
\newblock Breaking the curse of dimensionality, or how to use {SVD} in many
  dimensions.
\newblock {\em SIAM Journal on Scientific Computing}, 31(5):3744--3759, 2009.
\newblock \href {http://dx.doi.org/10.1137/090748330}
  {\path{doi:10.1137/090748330}}.

\bibitem{OSELEDETS2011}
I.~V. Oseledets.
\newblock Tensor-train decomposition.
\newblock {\em SIAM Journal on Scientific Computing}, 33(5):2295--2317, 2011.
\newblock \href {http://dx.doi.org/10.1137/090752286}
  {\path{doi:10.1137/090752286}}.

\bibitem{LARCHER2017}
T.~von Larcher and R.~Klein.
\newblock On identification of self-similar characteristics using the tensor
  train decomposition method with application to channel turbulence flow, 2017.
\newblock \href {http://arxiv.org/abs/1708.07780} {\path{arXiv:1708.07780}}.

\bibitem{XUE1996}
D.~Xue, Y.~Zhu, G.-X. Zhu, and X.~Yan.
\newblock Generalized {K}ronecker product and fractals.
\newblock {\em Proceedings of the SPIE}, 2644:75--78, 1996.
\newblock \href {http://dx.doi.org/10.1117/12.235499}
  {\path{doi:10.1117/12.235499}}.

\bibitem{HARDY2008}
A.~Hardy and W.~H. Steeb.
\newblock {\em Mathematical tools in computer graphics with {C}\#
  implementations}.
\newblock World Scientific, 2008.

\bibitem{LESKOVEC2010}
J.~Leskovec, D.~Chakrabarti, J.~Kleinberg, C.~Faloutsos, and Z.~Ghahramani.
\newblock Kronecker graphs: {A}n approach to modeling networks.
\newblock {\em Journal of Machine Learning Research}, 11:985--1042, 2010.

\bibitem{HAUSDORFF1918}
F.~Hausdorff.
\newblock Dimension und {\"a}u{\ss}eres {M}a{\ss}.
\newblock {\em Mathematische Annalen}, 79(1):157--179, 1918.
\newblock \href {http://dx.doi.org/10.1007/BF01457179}
  {\path{doi:10.1007/BF01457179}}.

\bibitem{LIEBOVITCH1989}
L.~S. Liebovitch and T.~Toth.
\newblock A fast algorithm to determine fractal dimensions by box counting.
\newblock {\em Physics Letters A}, 141(8):386--390, 1989.
\newblock \href {http://dx.doi.org/10.1016/0375-9601(89)90854-2}
  {\path{doi:10.1016/0375-9601(89)90854-2}}.

\bibitem{SARKAR1994}
N.~Sarkar and B.~B. Chaudhuri.
\newblock An efficient differential box-counting approach to compute fractal
  dimension of image.
\newblock {\em IEEE Transactions on Systems, Man, and Cybernetics},
  24(1):115--120, 1994.
\newblock \href {http://dx.doi.org/10.1109/21.259692}
  {\path{doi:10.1109/21.259692}}.

\bibitem{CANTOR1883}
G.~Cantor.
\newblock {\"U}ber die stetige {A}bbildung einer {L}inie auf ein
  {F}l{\"a}chenst{\"u}ck.
\newblock {\em Mathematische Annalen}, 21(4):545--591, 1883.
\newblock \href {http://dx.doi.org/10.1007/BF01446819}
  {\path{doi:10.1007/BF01446819}}.

\bibitem{SECELEAN2012}
N.~A. Secelean.
\newblock The existence of the attractor of countable iterated function
  systems.
\newblock {\em Mediterranean Journal of Mathematics}, 9(1):61--79, 2012.
\newblock \href {http://dx.doi.org/10.1007/s00009-011-0116-x}
  {\path{doi:10.1007/s00009-011-0116-x}}.

\bibitem{SIERPINSKI1916}
W.~Sierpi\'nski.
\newblock Sur une courbe cantorienne qui contient une image biunivoque et
  continue de toute courbe donn\'ee.
\newblock {\em Comptes Rendus Hebdomadaires des S\'eances de l'Acad\'emie des
  Sciences}, 162:629--632, 1916.

\bibitem{MENGER1916}
K.~Menger.
\newblock {A}llgemeine {R}{\"a}ume und cartesische {R}{\"a}ume. {I}.
\newblock {\em Proceedings Amsterdam}, 29:476--482, 1926.

\bibitem{KOLDA2009}
T.~G. Kolda and B.~W. Bader.
\newblock Tensor decompositions and applications.
\newblock {\em SIAM Review}, 51(3):455--500, 2009.
\newblock \href {http://dx.doi.org/10.1137/07070111X}
  {\path{doi:10.1137/07070111X}}.

\bibitem{COHEN2015b}
J.~E. Cohen.
\newblock About notations in multiway array processing, 2015.
\newblock \href {http://arxiv.org/abs/1511.01306} {\path{arXiv:1511.01306}}.

\bibitem{CICHOCKI2016}
A.~Cichocki, N.~Lee, I.~Oseledets, A.-H. Phan, Q.~Zhao, and D.~P. Mandic.
\newblock Tensor networks for dimensionality reduction and large-scale
  optimization: {P}art 1 low-rank tensor decompositions.
\newblock {\em Foundations and Trends in Machine Learning}, 9(4--5):249--429,
  2016.
\newblock \href {http://dx.doi.org/10.1561/2200000059}
  {\path{doi:10.1561/2200000059}}.

\bibitem{ARNOLD2013}
A.~Arnold and T.~Jahnke.
\newblock On the approximation of high-dimensional differential equations in
  the hierarchical {T}ucker format.
\newblock {\em BIT Numerical Mathematics}, 54(2):305--341, 2013.
\newblock \href {http://dx.doi.org/10.1007/s10543-013-0444-2}
  {\path{doi:10.1007/s10543-013-0444-2}}.

\bibitem{LUBICH2013}
C.~Lubich, T.~Rohwedder, R.~Schneider, and B.~Vandereycken.
\newblock Dynamical approximation by hierarchical {T}ucker and tensor-train
  tensors.
\newblock {\em SIAM Journal on Matrix Analysis and Applications},
  34(2):470--494, 2013.
\newblock \href {http://dx.doi.org/10.1137/120885723}
  {\path{doi:10.1137/120885723}}.

\bibitem{WHITE1992}
S.~R. White.
\newblock Density matrix formulation for quantum renormalization groups.
\newblock {\em Physical Review Letters}, 69(19):2863--2866, 1992.
\newblock \href {http://dx.doi.org/10.1103/PhysRevLett.69.2863}
  {\path{doi:10.1103/PhysRevLett.69.2863}}.

\bibitem{MEYER2009}
H.~D. Meyer, F.~Gatti, and G.~A. {Worth~(Eds.)}.
\newblock {\em Multidimensional quantum dynamics: {MCTDH} theory and
  applications}.
\newblock Wiley-VCH Verlag GmbH {\&} Co. KGaA, 2009.
\newblock \href {http://dx.doi.org/10.1002/9783527627400.ch3}
  {\path{doi:10.1002/9783527627400.ch3}}.

\bibitem{GELSS2016}
P.~Gel{\ss}, S.~Matera, and C.~Sch{\"u}tte.
\newblock Solving the master equation without kinetic {M}onte carlo: {T}ensor
  train approximations for a {CO} oxidation model.
\newblock {\em Journal of Computational Physics}, 314:489--502, 2016.
\newblock \href {http://dx.doi.org/10.1016/j.jcp.2016.03.025}
  {\path{doi:10.1016/j.jcp.2016.03.025}}.

\bibitem{KRESSNER2014}
D.~Kressner and F.~Macedo.
\newblock Low-rank tensor methods for communicating {M}arkov processes.
\newblock {\em Quantitative Evaluation of Systems, Lecture Notes in Computer
  Science}, 8657:25--40, 2014.

\bibitem{NOVIKOV2015}
A.~Novikov, D.~Podoprikhin, A.~Osokin, and D.~Vetrov.
\newblock Tensorizing neural networks.
\newblock In C.~Cortes, N.~D. Lawrence, D.~D. Lee, M.~Sugiyama, and R.~Garnett,
  editors, {\em Advances in Neural Information Processing Systems 28 (NIPS)},
  pages 442--450. Curran Associates, Inc., 2015.
\newblock \href {http://arxiv.org/abs/1509.06569v2}
  {\path{arXiv:1509.06569v2}}.

\bibitem{COHEN2015}
N.~Cohen, O.~Sharir, and A.~Shashua.
\newblock On the expressive power of deep learning: {A} tensor analysis, 2015.
\newblock \href {http://arxiv.org/abs/1509.05009} {\path{arXiv:1509.05009}}.

\bibitem{KLUS2016}
S.~Klus and C.~Sch{\"u}tte.
\newblock Towards tensor-based methods for the numerical approximation of the
  {P}erron--{F}robenius and {K}oopman operator.
\newblock {\em Journal of Computational Dynamics}, 3(2), 2016.
\newblock \href {http://dx.doi.org/10.3934/jcd.2016007}
  {\path{doi:10.3934/jcd.2016007}}.

\bibitem{KLUS2016b}
S.~Klus, P.~Gel{\ss}, S.~Peitz, and C.~Sch{\"u}tte.
\newblock Tensor-based dynamic mode decomposition, 2016.
\newblock \href {http://arxiv.org/abs/1606.06625} {\path{arXiv:1606.06625}}.

\bibitem{HACKBUSCH2012}
W.~Hackbusch.
\newblock {\em Tensor spaces and numerical tensor calculus}, volume~42 of {\em
  Springer Series in Computational Mathematics}.
\newblock Springer, 2012.
\newblock \href {http://dx.doi.org/10.1007/978-3-642-28027-6}
  {\path{doi:10.1007/978-3-642-28027-6}}.

\bibitem{KAZEEV2012}
V.~Kazeev and B.~Khoromskij.
\newblock Low-rank explicit {QTT} representation of the {L}aplace operator and
  its inverse.
\newblock {\em SIAM Journal on Matrix Analysis and Applications},
  33(3):742--758, 2012.
\newblock \href {http://dx.doi.org/10.1137/100820479}
  {\path{doi:10.1137/100820479}}.

\bibitem{VICSEK1992}
T.~Vicsek.
\newblock {\em Fractal growth phenomena}.
\newblock World Scientific, 1992.

\end{thebibliography}


\appendix

\section{Iterated function systems}

\subsection{Sierpinski carpet}\label{app: IFS - sierpinski}

The IFS corresponding to the Sierpinski carpet is given by
\begin{equation*}
  \left\lbrace 
    \begin{array}{ll} 
      f_1 (x,y) = \frac{1}{3} (x,y),  					& f_2 (x,y) = \frac{1}{3} (x,y) + (\frac{1}{3}, 0), \\[0.1cm] 
      f_3 (x,y) = \frac{1}{3} (x,y) + (\frac{2}{3}, 0), 		& f_4 (x,y) = \frac{1}{3} (x,y) + (0, \frac{1}{3}), \\[0.1cm] 
      f_5 (x,y) = \frac{1}{3} (x,y) + (\frac{2}{3}, \frac{1}{3}), 	& f_6 (x,y) = \frac{1}{3} (x,y) + (0, \frac{2}{3}), \\[0.1cm] 
      f_7 (x,y) = \frac{1}{3} (x,y) + (\frac{1}{3}, \frac{2}{3}), 	& f_8 (x,y) = \frac{1}{3} (x,y) + (\frac{2}{3}, \frac{2}{3})
    \end{array}
  \right\rbrace.
\end{equation*}

\subsection{Menger sponge}\label{app: IFS - menger}

The IFS corresponding to the Menger sponge is given by
\begin{equation*}
  \left\lbrace 
    \begin{array}{ll} 
      f_1 (x,y) = \frac{1}{3} (x,y,z),  					& f_2 (x,y) = \frac{1}{3} (x,y,z) + (\frac{1}{3}, 0, 0), \\[0.1cm] 
      f_3 (x,y) = \frac{1}{3} (x,y,z) + (\frac{2}{3}, 0, 0), 			& f_4 (x,y) = \frac{1}{3} (x,y,z) + (0, \frac{1}{3}, 0), \\[0.1cm] 
      f_5 (x,y) = \frac{1}{3} (x,y,z) + (\frac{2}{3}, \frac{1}{3}, 0), 		& f_6 (x,y) = \frac{1}{3} (x,y,z) + (0,\frac{2}{3}, 0), \\[0.1cm] 
      f_7 (x,y) = \frac{1}{3} (x,y,z) + (\frac{1}{3}, \frac{2}{3}, 0), 		& f_8 (x,y) = \frac{1}{3} (x,y,z) + (\frac{2}{3}, \frac{2}{3},0), \\[0.1cm]
      f_9 (x,y) = \frac{1}{3} (x,y,z) + (0, 0, \frac{1}{3}),  			& f_{10} (x,y) = \frac{1}{3} (x,y,z) + (\frac{2}{3}, 0, \frac{1}{3}), \\[0.1cm] 
      f_{11} (x,y) = \frac{1}{3} (x,y,z) + (0, \frac{2}{3}, \frac{1}{3}), 	& f_{12} (x,y) = \frac{1}{3} (x,y,z) + (\frac{2}{3}, \frac{2}{3}, \frac{1}{3}), \\[0.1cm] 
      f_{13} (x,y) = \frac{1}{3} (x,y,z) + (0, 0, \frac{2}{3}), 	& f_{14} (x,y) = \frac{1}{3} (x,y,z) + (\frac{1}{3}, 0, \frac{2}{3}), \\[0.1cm] 
      f_{15} (x,y) = \frac{1}{3} (x,y,z) + (\frac{2}{3}, 0, \frac{2}{3}), 	& f_{16} (x,y) = \frac{1}{3} (x,y,z) + (0, \frac{1}{3}, \frac{2}{3}), \\[0.1cm]
      f_{17} (x,y) = \frac{1}{3} (x,y,z) + (\frac{2}{3}, \frac{1}{3}, \frac{2}{3}), 	& f_{18} (x,y) = \frac{1}{3} (x,y,z) + (0, \frac{2}{3}, \frac{2}{3}), \\[0.1cm] 
      f_{19} (x,y) = \frac{1}{3} (x,y,z) + (\frac{1}{3}, \frac{2}{3}, \frac{2}{3}), 	& f_{20} (x,y) = \frac{1}{3} (x,y,z) + (\frac{2}{3}, \frac{2}{3}, \frac{2}{3})
    \end{array}
  \right\rbrace.
\end{equation*}

\section{Proofs}

\subsection{Non-zero entries of the defining tensor of the $d$-dimensional multisponge}\label{app: non-zero multisponge}

\begin{theorem}
 The number $m$ of non-zero entries of the defining tensor as given in \eqref{eq: TT 3} is $(d+2)\cdot 2^{d-1} $.
\end{theorem}
\begin{proof}
 Denote the defining tensor of the $d$-dimensional multisponge by $\mathbf{T}$, i.e.
 \begin{equation*}
 \begin{split}
  \mathbf{T} &= \sum_{k_0=1}^{r_0} \cdots  \sum_{k_d=1}^{r_d} \left( \mathbf{T}^{(1)} \right)_{k_0,:,k_1} \otimes \dots \otimes \left( \mathbf{T}^{(d)} \right)_{k_{d-1},:,k_d} \\
  & = \begin{bmatrix} \begin{pmatrix} 1 \\ 1 \\ 1 \end{pmatrix} & \begin{pmatrix} 1 \\ 0 \\ 1 \end{pmatrix} \end{bmatrix} \otimes \begin{bmatrix} \begin{pmatrix} 1 \\ 0 \\ 1 \end{pmatrix} & \begin{pmatrix} 0 \\ 0 \\ 0 \end{pmatrix} \\[0.5cm] \begin{pmatrix} 0 \\ 1 \\ 0 \end{pmatrix} & \begin{pmatrix} 1 \\ 0 \\ 1 \end{pmatrix} \end{bmatrix} \otimes \dots \otimes \begin{bmatrix} \begin{pmatrix} 1 \\ 0 \\ 1 \end{pmatrix} & \begin{pmatrix} 0 \\ 0 \\ 0 \end{pmatrix} \\[0.5cm] \begin{pmatrix} 0 \\ 1 \\ 0 \end{pmatrix} & \begin{pmatrix} 1 \\ 0 \\ 1 \end{pmatrix} \end{bmatrix} \otimes \begin{bmatrix} \begin{pmatrix} 1 \\ 0 \\ 1 \end{pmatrix} \\[0.5cm] \begin{pmatrix} 0 \\ 1 \\ 0 \end{pmatrix} \end{bmatrix},
 \end{split}
 \end{equation*}
 with $r_0 = r_d = 1$ and $r_1 = \dots = r_{d-1} = 2$. In order to count the non-zero entries of $\mathbf{T}$, we simply sum the over all dimensions since $\mathbf{T}$ is a binary tensor. That is, we obtain
 \begin{equation*}
 \begin{split}
  \sum_{x_1 = 1}^3 \dots \sum_{x_d = 1}^3 \mathbf{T}_{x_1 , \dots , x_d} & = \sum_{x_1 = 1}^3 \dots \sum_{x_d = 1}^3 \sum_{k_0=1}^{r_0} \cdots  \sum_{k_d=1}^{r_d} \left( \mathbf{T}^{(1)} \right)_{k_0,x_1,k_1} \otimes \dots \otimes \left( \mathbf{T}^{(d)} \right)_{k_{d-1},x_d,k_d}\\
  & = \left( \sum_{x_1 = 1}^3 \mathbf{T}^{(1)}_{:,x_1 , :}\right) \cdot \dots \cdot \left( \sum_{x_d = 1}^3 \mathbf{T}^{(d)}_{:,x_d , :}\right) \\
  & = \begin{pmatrix} 3 & 2 \end{pmatrix} \cdot \begin{pmatrix} 2 & 0 \\ 1 & 2 \end{pmatrix}^{d-2} \cdot   \begin{pmatrix} 2 \\ 1 \end{pmatrix} \\
  & = \begin{pmatrix} 3 & 2 \end{pmatrix} \cdot \begin{pmatrix} 2^{d-2} & 0 \\ (d-2)\cdot 2^{d-3} & 2^{d-2} \end{pmatrix}^{d-2} \cdot   \begin{pmatrix} 2 \\ 1 \end{pmatrix} \\
  & = (d+2) \cdot 2^{d-1}. \qedhere
 \end{split}
 \end{equation*}

\end{proof}

\subsection{Volume of the $d$-dimensional multisponge}\label{app: volume multisponge}

\begin{theorem}
 The volume of $C^d_\infty$, $d \geq 3$, is 0.
\end{theorem}
\begin{proof}
 By induction, we show that the volume $V_k$ given in \eqref{eq: volume multisponge} approaches 0 for $k \rightarrow \infty$, i.e.~we show that
 \begin{equation*}
  (d+2)\cdot 2^{d-1} < 3^d,
  \tag{$\ast$}
 \end{equation*}
 for $d \geq 3$. Let $d = 3$, then
 \begin{equation*}
  (d+2)\cdot 2^{d-1} = 20 < 27 = 3^d. 
 \end{equation*}
 If ($\ast$) holds for some $d \geq 3$, then the statement also holds for $d + 1$ because
 \begin{equation*}
  (d+3)\cdot 2^{d} = (d+2)\cdot 2^{d-1}\cdot 2 + 2^d  < 3^d \cdot 2 + 2^d < 3^d \cdot 2 + 3^d = 3^{d+1}. \qedhere
 \end{equation*}

\end{proof}

\subsection{Connectedness of the $d$-dimensional multisponge}\label{app: connectedness multisponge}

\begin{theorem}
	The $d$-dimensional multisponge $C^d_\infty$, $d \geq 3$, is connected.
\end{theorem}
\begin{proof}
In order to prove the connectedness of the limit $C^d_\infty$, we have to show that the defining tensor of the $d$-dimensional multisponge is connected. That is, for two given non-zero entries $\mathbf{T}^d_{x_1 , \dots , x_d}$ and $\mathbf{T}^d_{y_1 , \dots , y_d} $ of the defining tensor $\mathbf{T}^d$ with
\begin{equation*}
\mathbf{T}^d = \begin{bmatrix} \begin{pmatrix} 1 \\ 1 \\ 1 \end{pmatrix} & \begin{pmatrix} 1 \\ 0 \\ 1 \end{pmatrix} \end{bmatrix} \otimes \begin{bmatrix} \begin{pmatrix} 1 \\ 0 \\ 1 \end{pmatrix} & \begin{pmatrix} 0 \\ 0 \\ 0 \end{pmatrix} \\[0.5cm] \begin{pmatrix} 0 \\ 1 \\ 0 \end{pmatrix} & \begin{pmatrix} 1 \\ 0 \\ 1 \end{pmatrix} \end{bmatrix} \otimes \dots \otimes \begin{bmatrix} \begin{pmatrix} 1 \\ 0 \\ 1 \end{pmatrix} & \begin{pmatrix} 0 \\ 0 \\ 0 \end{pmatrix} \\[0.5cm] \begin{pmatrix} 0 \\ 1 \\ 0 \end{pmatrix} & \begin{pmatrix} 1 \\ 0 \\ 1 \end{pmatrix} \end{bmatrix} \otimes \begin{bmatrix} \begin{pmatrix} 1 \\ 0 \\ 1 \end{pmatrix} \\[0.5cm] \begin{pmatrix} 0 \\ 1 \\ 0 \end{pmatrix} \end{bmatrix}
\end{equation*}
there has to exist a path from the index tuple $(x_1 , \dots , x_d)$ to the index tuple $(y_1 , \dots , y_d)$. Here, a path means a sequence of index tuples such that only one index $x_i$, $i \in \{1, \dots , d\}$, is increased/decreased by $1$ at each step of the sequence and the corresponding entries of $\mathbf{T}^d$ are equal to $1$ for all elements of the sequence. The proof is done by induction:

For $d=3$, we obtain the defining tensor of the Menger sponge. As one can see in Figure \ref{fig: menger} (b), this tensor is connected in the above sense. Now, we assume that the assertion holds for some $d \geq 3$, i.e.~the tensor $\mathbf{T}^d$ is connected. The defining tensor $\mathbf{T}^{d+1}$ satisfies
\begin{equation*}
\begin{split}
\mathbf{T}^{d+1}_{:, \dots , :, 1, :} & =  \mathbf{T}^{d+1}_{:, \dots , :, 3, :} \\ 
& =\begin{bmatrix} \begin{pmatrix} 1 \\ 1 \\ 1 \end{pmatrix} & \begin{pmatrix} 1 \\ 0 \\ 1 \end{pmatrix} \end{bmatrix} \otimes \begin{bmatrix} \begin{pmatrix} 1 \\ 0 \\ 1 \end{pmatrix} & \begin{pmatrix} 0 \\ 0 \\ 0 \end{pmatrix} \\[0.5cm] \begin{pmatrix} 0 \\ 1 \\ 0 \end{pmatrix} & \begin{pmatrix} 1 \\ 0 \\ 1 \end{pmatrix} \end{bmatrix} \otimes \dots \otimes \begin{bmatrix} \begin{pmatrix} 1 \\ 0 \\ 1 \end{pmatrix} & \begin{pmatrix} 0 \\ 0 \\ 0 \end{pmatrix} \\[0.5cm] \begin{pmatrix} 0 \\ 1 \\ 0 \end{pmatrix} & \begin{pmatrix} 1 \\ 0 \\ 1 \end{pmatrix} \end{bmatrix} \otimes  \begin{bmatrix} \begin{pmatrix} 1  \end{pmatrix} & \begin{pmatrix} 0  \end{pmatrix} \\[0.1cm] \begin{pmatrix} 0  \end{pmatrix} & \begin{pmatrix} 1  \end{pmatrix} \end{bmatrix} \otimes \begin{bmatrix} \begin{pmatrix} 1 \\ 0 \\ 1 \end{pmatrix} \\[0.5cm] \begin{pmatrix} 0 \\ 1 \\ 0 \end{pmatrix} \end{bmatrix}\\
& =\begin{bmatrix} \begin{pmatrix} 1 \\ 1 \\ 1 \end{pmatrix} & \begin{pmatrix} 1 \\ 0 \\ 1 \end{pmatrix} \end{bmatrix} \otimes \begin{bmatrix} \begin{pmatrix} 1 \\ 0 \\ 1 \end{pmatrix} & \begin{pmatrix} 0 \\ 0 \\ 0 \end{pmatrix} \\[0.5cm] \begin{pmatrix} 0 \\ 1 \\ 0 \end{pmatrix} & \begin{pmatrix} 1 \\ 0 \\ 1 \end{pmatrix} \end{bmatrix} \otimes \dots \otimes \begin{bmatrix} \begin{pmatrix} 1 \\ 0 \\ 1 \end{pmatrix} & \begin{pmatrix} 0 \\ 0 \\ 0 \end{pmatrix} \\[0.5cm] \begin{pmatrix} 0 \\ 1 \\ 0 \end{pmatrix} & \begin{pmatrix} 1 \\ 0 \\ 1 \end{pmatrix} \end{bmatrix}  \otimes \begin{bmatrix} \begin{pmatrix} 1 \\ 0 \\ 1 \end{pmatrix} \\[0.5cm] \begin{pmatrix} 0 \\ 1 \\ 0 \end{pmatrix} \end{bmatrix}\\
&= \mathbf{T}^d.
\end{split}
\end{equation*}
Furthermore, 
\begin{equation*}
\mathbf{T}^{d+1}_{x_1, \dots , x_{d+1}} = 1,
\end{equation*}
if each index $x_i$ is equal to $1$ or $3$. Considering $\mathbf{T}^{d+1}_{:, \dots , :, 2, :}$, it holds that 
\begin{equation*}
\mathbf{T}^{d+1}_{x_1, \dots , x_{d-1}, 2, x_{d+1}} = 1,
\end{equation*}
if (and only if) $x_i = 1 $ or $x_i = 3$ for all $i \in \{1, \dots , d+1\}\backslash\{d\}$. Thus, each index tuple $(x_1 , \dots , x_{d-1}, 2, x_{d+1})$ with $\mathbf{T}^{d+1}_{x_1, \dots , x_{d-1}, 2, x_{d+1}} = 1$ is connected to $(x_1 , \dots , x_{d-1}, 1, x_{d+1})$ as well as to $(x_1 , \dots , x_{d-1}, 3, x_{d+1})$. In other words, every non-zero entry of $\mathbf{T}^{d+1}_{:, \dots , :, 2, :}$ is connected to $\mathbf{T}^{d+1}_{:, \dots , :, 1, :}$ and $\mathbf{T}^{d+1}_{:, \dots , :, 3, :}$. Since these both tensors are equal to $\mathbf{T}^d$ and therefore are internally connected, the whole tensor $\mathbf{T}^{d+1}$ is connected. \qedhere
\end{proof}

\section{Defining matrices for RGB fractals}\label{app: color}

\noindent Defining matrices corresponding to Figure \ref{fig: color a}:
\begin{equation*}
\begin{split}
M_R  = \begin{pmatrix} \frac{1}{2} & 1 & \frac{1}{2} \\[0.1cm] 1 & \frac{1}{2} & 1 \\[0.1cm] \frac{1}{2} & 1 & \frac{1}{2} \end{pmatrix}, \quad
M_G = \begin{pmatrix} \frac{3}{4} & 1 & \frac{3}{4} \\[0.1cm] 1 & 1 & 1 \\[0.1cm] \frac{3}{4} & 1 & \frac{3}{4} \end{pmatrix}, \quad
M_B = \begin{pmatrix} 1 & \frac{3}{4} & 1 \\[0.1cm] \frac{3}{4} & 1 & \frac{3}{4} \\[0.1cm] 1 & \frac{3}{4} & 1 \end{pmatrix}.
\end{split}
\end{equation*}

\noindent Defining matrices corresponding to Figure \ref{fig: color b}:
\begin{equation*}
M_R = \begin{pmatrix} \frac{1}{2} & \frac{3}{4} &\frac{3}{4} &\frac{1}{2} \\[0.1cm] \frac{3}{4} & 1 & 1 & \frac{3}{4} \\[0.1cm] \frac{3}{4} & 1 & 1 & \frac{3}{4} \\[0.1cm] \frac{1}{2} & \frac{3}{4} & \frac{3}{4} & \frac{1}{2} \end{pmatrix}, \quad
M_G = \begin{pmatrix} 1 &\frac{1}{2} &\frac{1}{2} &1 \\[0.1cm] \frac{1}{2} &\frac{3}{4} &\frac{3}{4} &\frac{1}{2} \\[0.1cm] \frac{1}{2} &\frac{3}{4} &\frac{3}{4} &\frac{1}{2} \\[0.1cm] 1 &\frac{1}{2} &\frac{1}{2} &1\end{pmatrix}, \quad
M_B = \begin{pmatrix} \frac{3}{4} & 1 & 1 &\frac{3}{4} \\[0.1cm] 1 &\frac{1}{2} &\frac{1}{2} &1 \\[0.1cm] 1 &\frac{1}{2} &\frac{1}{2} &1 \\[0.1cm] \frac{3}{4} &1 &1 &\frac{3}{4} \end{pmatrix}.
\end{equation*}

\noindent Defining matrices corresponding to Figure \ref{fig: color c}:
\begin{equation*}
M_R = \begin{pmatrix} \frac{1}{4} & \frac{1}{2} & 1 & \frac{1}{2} &\frac{1}{4} \\[0.1cm] \frac{1}{2} &1 &1 &1 &\frac{1}{2} \\[0.1cm] 1 &1 &\frac{1}{2} &1 &1 \\[0.1cm] \frac{1}{2} &\frac{1}{2} &\frac{1}{4} &\frac{1}{2} &\frac{1}{2} \\[0.1cm] \frac{1}{2} &\frac{1}{4} &\frac{1}{4} &\frac{1}{4} &\frac{1}{2}\end{pmatrix}, \quad
M_G = \begin{pmatrix} \frac{1}{4} &\frac{1}{4} &\frac{1}{2} &\frac{1}{4} &\frac{1}{4} \\[0.1cm] \frac{1}{4} &\frac{1}{2}  &1 &\frac{1}{2}  &\frac{1}{4} \\[0.1cm] \frac{1}{2} &1 &1 &1 &\frac{1}{2} \\[0.1cm] 1 &1 &\frac{1}{2} &1 &1 \\[0.1cm] \frac{1}{2} &\frac{1}{2} &\frac{1}{4} &\frac{1}{2} &\frac{1}{2} \end{pmatrix}, \quad 
M_B = \begin{pmatrix} \frac{1}{4} &\frac{1}{4} &\frac{1}{4} &\frac{1}{4} &\frac{1}{4} \\[0.1cm] \frac{1}{4} &\frac{1}{4} &\frac{1}{2} &\frac{1}{4} &\frac{1}{4} \\[0.1cm] \frac{1}{4} &\frac{1}{2}  &1 &\frac{1}{2}  &\frac{1}{4} \\[0.1cm] \frac{1}{2} &1 &1 &1 &\frac{1}{2} \\[0.1cm] 1 &1 &\frac{1}{2} &1 &1\end{pmatrix}.
\end{equation*}
\end{document}